\documentclass[a4paper,10pt]{amsart}

\title[Cartan geometries on complex surfaces]%
{Holomorphic Cartan geometries on uniruled surfaces}
\date{15 April 2011}

\author{Benjamin McKay}

\address{School of Mathematical Sciences, 
University College Cork, Cork, Ireland}
\email{b.mckay@ucc.ie}

\thanks{It is a pleasure to thank
Sorin Dumitrescu for helpful conversations
on the problems solved in this paper,
and for inviting me to 
the
Laboratoire de Math{\'e}matiques J.A. Dieudonn{\'e} at the University of Nice Sophia--Antipolis
where this paper was written.}
\keywords{Cartan connection, complex surface}
\date{\today}

\usepackage{version}

\includeversion{unabridged}
\excludeversion{abridged}

\usepackage[usenames,dvipsnames]{color}
\usepackage[latin1]{inputenc}
\usepackage{etex}
\usepackage{tikz}
\usetikzlibrary{matrix,arrows}
\usepackage{longtable}
\usepackage{array}
\usepackage{ragged2e}
\usepackage{varioref}
\usepackage{amsfonts}
\usepackage{amsmath}
\usepackage{braket}
\usepackage{amssymb}
\usepackage{amsthm}
\usepackage{booktabs}
\usepackage[all]{xy}
\usepackage{fontenc}
\usepackage{ifthen}
\usepackage{verbatim}
\usepackage[pdftex]{hyperref}
\newtheorem{theorem}{Theorem}

\newtheorem{lemma}{Lemma}
\newtheorem{proposition}{Proposition}
\theoremstyle{remark}

\newtheorem{definition}{Definition}
\newtheorem{example}{Example}
\newtheorem{remark}{Remark}

\newcommand{\C}[1]{\ensuremath{\mathbb{C}^{#1}}}

\newcommand{\Z}[1]{\ensuremath{\mathbb{Z}^{#1}}}

\newcommand{\OO}[1]{%
  \ensuremath{%
    \mathcal{O}%
    \ifthenelse{\equal{#1}{0}}%
      {}%
      {\left({#1}\right)}%
  }%
}%
\newcommand{\OOp}[2]{
  \ensuremath{
    \mathcal{O}
    \ifthenelse{\equal{#1}{0}}
      {}
      {\left({#1}\right)}
    \ifthenelse{\equal{#2}{1}}
      {}
      {^{\oplus{#2}}}
  }
}
\newcommand{\Proj}[1]{\ensuremath{\mathbb{P}^{#1}}}
\newcommand{\Sym}[2]{\ensuremath{\operatorname{Sym}^{#1}\left(#2\right)}}
\newcommand{\GL}[1]{\ensuremath{\operatorname{GL}\left(#1\right)}}

\newcommand{\SL}[1]{\ensuremath{\operatorname{SL}\left(#1\right)}}
\newcommand{\PSL}[1]{\ensuremath{\mathbb{P}\operatorname{SL}\left(#1\right)}}

\newcommand{\nForms}[2]{\ensuremath{\Omega^{#1} \left ( {#2} \right)}}
\newcommand{\Lm}[2]{\ensuremath{\Lambda^{#1}\left({#2}\right)}}
\newcommand{\Cohom}[2]{\ensuremath{H^{#1}\left({#2}\right)}}
\DeclareMathOperator{\Ad}{Ad}

\newcommand{\hook}{\ensuremath{\mathbin{ \hbox{\vrule height1.4pt
        width4pt depth-1pt \vrule height4pt width0.4pt depth-1pt}}}}
\newcommand{\SO}[1]{\ensuremath{\operatorname{SO}\left(#1\right)}}

\newcommand{\pd}[2]{\frac{\partial #1}{\partial #2}}
\newcommand{\map}[3][:]%
{\ensuremath{\ifthenelse{\equal{#1}{:}}{}{{#1} \colon}{#2} \to {#3}}}
\newcommand{\mapto}[3][:]%
{\ensuremath{\ifthenelse{\equal{#1}{:}}{}{{#1} \colon}{#2} \mapsto {#3}}}
\newcommand{\Hom}[2]{\ensuremath{\operatorname{Hom}\left({#1},{#2}\right)}}
\newcommand{\homotopygp}[2]%
{\ensuremath{\pi_{#1}\left({#2}\right)}}
\newcommand{\fundgp}[1]%
{\ensuremath{\homotopygp{1}{#1}}}

\newcommand{\Lie}[1]{\ensuremath{\mathfrak{#1}}}
\newcommand{\LieSL}[1]%
{\ensuremath{\Lie{sl}\left({#1}\right)}}
\newcommand{\LieGL}[1]%
{\ensuremath{\Lie{gl}\left({#1}\right)}}

\newcommand{\LieG}{\Lie{g}}
\newcommand{\LieH}{\Lie{h}}
\newcommand{\LieK}{\Lie{k}}

\newcommand{\Bihol}[1]%
{\ensuremath{\operatorname{Bihol}\left(#1\right)}}
\newcommand{\Aut}[1]%
{\ensuremath{\operatorname{Aut}\left(#1\right)}}
\newcommand{\AutP}[1]%
{\ensuremath{\operatorname{Aut}_{\pi}\left(#1\right)}}
\newcommand{\ExtAut}[1]%
{\ensuremath{\operatorname{Aut}'\left(#1\right)}}

\newcounter{remarkCounter}
\newcolumntype{A}{>{$}c<{$}}
\newcolumntype{Q}{>{$}l<{$}}
\newcolumntype{V}{>{$}r<{$}}
\newcommand{\betweenEntries}{%
\\* \addlinespace[10pt]
}%
\newcommand{\otoprule}%
{\midrule[\heavyrulewidth]\addlinespace[5pt]}

{%
\small
\begin{center}
\begin{longtable}{#1}
}%
{%
\end{longtable}
\end{center}
}%

\begin{document}

\begin{abstract} 
\begin{unabridged}  
We classify holomorphic Cartan geometries
on every compact complex curve, and on
every compact complex surface which
contains a rational curve.
\end{unabridged}  
\begin{abridged}  
We classify holomorphic Cartan geometries
on every compact complex surface which
contains a rational curve.
\end{abridged}  
\end{abstract}

\maketitle
\tableofcontents

\section{Introduction}

In my work with Indranil Biswas
\cite{Biswas/McKay:2010}
it became clear that the classification of
holomorphic Cartan geometries on uniruled 
compact complex surfaces was in sight. 
In this paper, I find that classification. 
The classification is essential 
for a paper in progress in which Sorin 
Dumitrescu and I will survey the
holomorphic locally homogeneous structures on
all compact complex surfaces.
The main result in the paper you have
before you is that
if a compact complex surface $S$ bears a
holomorphic
Cartan geometry, and contains a rational curve,
then the Cartan geometry is flat, and
either $S$ is a rational homogeneous
variety with its standard flat holomorphic
Cartan geometry, or $S$ is a flat
$\Proj{1}$-bundle $S \to C$
and the holomorphic Cartan geometry is 
induced by a locally homogeneous geometric 
structure on $C$. 
We will see explicit examples of 
complex manifolds on which the
moduli stack of holomorphic Cartan
geometries with a fixed model is not
a complex analytic space (or orbispace). We also
see in explicit examples 
that deformations of holomorphic
Cartan geometries can give rise
to nontrivial deformations of the total
space of the Cartan geometry as
a holomorphic principal bundle.

{}%
\begin{unabridged}%
To complete the picture,
I find all locally homogeneous
geometric structures on compact complex curves.
I also prove several elementary results for
which I don't know a reference.
\end{unabridged}%
\begin{abridged}%
A more detailed version of this paper appears
on the arXiv.
\end{abridged}
\section{Definitions}

\begin{abridged}
We will refer to $(G,X)$-structures \cite{Goldman:2010} 
where $X=G/H$ is homogeneous as $G/H$-structures.
\end{abridged}

\begin{unabridged}

\subsection{%
\texorpdfstring%
{Definition of $G/H$-structures}
{Definition of G/H-structures}
}%

Suppose that $G$ is a Lie group and that $H \subset G$
is a closed subgroup.
\begin{definition}
A \emph{$G/H$-chart} \cite{Goldman:2010} on a manifold $M$
is a local diffeomorphism from an open
subset of $M$ to an open subset of $G/H$.
\end{definition}

\begin{definition}
Two $G/H$-charts $f_0$ and $f_1$ 
on a manifold are \emph{compatible}
if there is some element $g \in G$
so that $f_1 = g f_0$ where both $f_0$
and $f_1$ are defined.
\end{definition}

\begin{definition}
A \emph{$G/H$-atlas} on a manifold $M$
is a collection of mutually compatible
$G/H$-charts.
\end{definition}

\begin{definition}
A $G/H$-structure on a manifold $M$ 
is a maximal $G/H$-atlas.
\end{definition}

\begin{definition}
A $G/H$-structure is \emph{effective}
if $G$ acts faithfully on $G/H$.
\end{definition}

\subsection{Pulling back}

\begin{definition}
If $\map[F]{M_0}{M_1}$ is a local diffeomorphism,
and $H \subset G$ a closed subgroup of a Lie group,
then every $G/H$-structure on $M_1$ has a \emph{pullback
structure} on $M_0$, whose charts are precisely
the compositions $F \circ f$, for $f$ a chart
of the $G/H$-structure. Conversely, if
$F$ is a regular covering map, and $M_0$
has a $G/H$-structure which is invariant
under the deck transformations, then
it induces a $G/H$-structure on $M_1$.
\end{definition}

\subsection{Developing maps and holonomy morphisms}

\begin{definition}
Suppose that $\left(M,m_0\right)$ 
is a pointed manifold,
with universal covering space 
$\left(\tilde{M},\tilde{m}_0\right)$.
Suppose that $H \subset G$ is a closed
subgroup of a Lie group.
Let $X=G/H$ and $x_0=1 \cdot H \in X$.
A $G/H$-\emph{developing system}
is a pair $\left(\delta,h\right)$
of maps, where
\[
\map[\delta]%
{\left(\tilde{M},\tilde{m}_0\right)}%
{\left(X,x_0\right)}
\]
is a local diffeomorphism
and 
\[
\map[h]{\fundgp{M}}{G}
\]
is a group homomorphism
so that
\[
\delta\left(\gamma \tilde{m}\right)
=
h\left(\gamma\right)\delta\left(\tilde{m}\right),
\]
for every $\gamma \in \fundgp{M}$
and $\tilde{m} \in \tilde{M}$.
The map $\delta$ is called the
\emph{developing map}, and the
morphism $h$ is called the 
\emph{holonomy morphism}
of the developing system.
\end{definition}

\begin{definition}
Denote the universal covering map of a
pointed manifold $\left(M,m_0\right)$
as 
\[
\map[\pi_M]{\left(\tilde{M},\tilde{m}_0\right)}%
{\left(M,m_0\right)}.
\] 
Given a $G/H$-developing system $\left(\delta,h\right)$
on a manifold $M$, the \emph{induced $G/H$-structure} 
on $M$ is the one whose charts are all maps $f$
so that $\delta = f \circ \pi_M$.
\end{definition}

\begin{remark}\label{remark:Goldman}
Conversely, it is well known \cite{Goldman:2010} that
if $G$ acts faithfully on $G/H$ then every
$G/H$-structure is induced by a developing system
$\left(\delta,h\right)$, which is uniquely
determined (after choice of a point
$m_0 \in M$ about which to develop) up to conjugacy:
\[
\mapto{\left(\delta,h\right)}%
{\left(g\delta,\Ad(g)h\right)}
\]
\end{remark}

\subsection{Inducing structures from other structures}

\begin{definition}
Suppose that 
$\map[\varphi]{G_0}{G_1}$
is a morphism of Lie groups,
that $H_0 \subset G_0$
and $H_1 \subset G_1$
are closed subgroups,
and that 
$\varphi\left(H_0\right) \subset H_1$.
Define a smooth map
\[
\map[\Phi]{G_0/H_0}{G_1/H_1}
\]
by
\[
\Phi\left(g_0H_0\right)
=
\varphi\left(g_0\right)H_1.
\]
Suppose also that 
\[
\map[\varphi'(1)]{\LieG_0/\LieH_0}{\LieG_1/\LieH_1}
\]
is a linear isomorphism.
Then clearly $\Phi$ is
a local diffeomorphism. 
If $M$ is equipped with a $G_0/H_0$-chart $f$,
then $\Phi \circ f$ is clearly a $G_1/H_1$-chart.
A $G_0/H_0$-structure $\left\{f_{\alpha}\right\}$
has \emph{induced} $G_1/H_1$-structure
$\left\{\Phi \circ f_{\alpha}\right\}$.
Every developing system $\left(\delta,h\right)$
on $M$ has induced developing system
$\left(\Phi \circ \delta,\varphi \circ h\right)$.
\end{definition}

\end{unabridged}

\subsection{Quotienting by the kernel}

\begin{definition}[Sharpe \cite{Sharpe:1997}]
If $H \subset G$ is a closed subgroup
of a Lie group, the \emph{kernel}
of the pair $\left(G,H\right)$
is
\[
K = \bigcap_{g \in G} gHg^{-1},
\]
i.e. the largest subgroup of $H$
which is normal in $G$.
The kernel is precisely the set
of elements of $G$ which act trivially
on $G/H$.
\end{definition}

\begin{example}
If $(G,H)$ has kernel $K$, we can let
$\bar{G}=G/K$, $\bar{H}=H/K$,
make the obvious morphism 
$\mapto{g \in G}{\bar{g}=gK \in \bar{G}}$,
and then clearly $G/H=\bar{G}/\bar{H}$.
Every $G/H$-structure then has induced 
$\bar{G}/\bar{H}$-structure, with the
same charts, called the \emph{induced effective 
structure}. Any developing map $\delta$
and holonomy morphism $h$
for the $G/H$-structure
gives the obvious developing map
$\bar\delta=\delta$ and holonomy morphism $\bar{h}$:
the composition of $h$ with $\map{G}{\bar{G}}$.
\end{example}

\begin{unabridged}

\begin{example}
Continuing with the same example,
if there is a morphism $\map[h]{\fundgp{M}}{G}$
so that $\bar{h}$ is the composition
\[
\xymatrix{
& G \ar[dd] \\
\fundgp{M} \ar[ur]^{h} \ar[dr]^{\bar{h}} \\
& \bar{G}
}
\]
then clearly we can let $\delta=\bar{\delta}$
and $\left(\delta,h\right)$ is a developing
system for $\Gamma$ as a $G/H$-structure.
Such a morphism $h$ 
need not exist and (as we will see below) 
it also need not be unique.
\end{example}

\end{unabridged}

\begin{abridged}
A Cartan geometry modelled on a homogeneous
space $G/H$ will also be called a $G/H$-geometry.
\end{abridged}

\begin{unabridged}

\subsection{Definition of Cartan geometries}
\begin{definition}
If \(E \to M\) is a principal right \(G\)-bundle,
we will write the right \(G\)-action as \(r_g e = eg\), where \(e\in E\)
and \(g\in G\).
\end{definition}

Throughout we use the convention that principal bundles
are right principal bundles.

\begin{definition}\label{def:CartanConnection}
Let \(H \subset G\) be a closed subgroup of a Lie group,
with Lie algebras \(\LieH \subset \LieG\). A \(G/H\)-geometry, or
\emph{Cartan geometry}
modelled on \(G/H\), on a manifold \(M\) is a choice of
$C^\infty$ principal $H$-bundle $E \to M$, and smooth
1-form $\omega \in
\nForms{1}{E} \otimes \LieG$ called the 
\emph{Cartan connection},
which satisifies all of the following conditions:
\begin{enumerate}
\item
\(
r_h^* \omega = \Ad_h^{-1} \omega
\) for all \(h \in H\).
\item
$\map[\omega_e]{T_e E}{\LieG}$ is a linear isomorphism at each point
$e \in E$.
\item
For each $A \in \LieG$, define a vector field $\vec{A}$ on $E$ by
the equation $\vec{A} \hook \omega = A$. Then the vector fields
$\vec{A}$ for $A \in \LieH$ generate the $H$-action on $E$.
\end{enumerate}
\end{definition}
Sharpe \cite{Sharpe:2002} gives an introduction to Cartan geometries.
\begin{example}
The principal $H$-bundle $G \to G/H$ is a Cartan geometry, with
Cartan connection $\omega=g^{-1} \, dg$ the left invariant
Maurer--Cartan 1-form on $G$; this geometry is called the
\emph{model Cartan geometry}.
\end{example}
\begin{definition}
An \emph{isomorphism} of $G/H$-geometries $E_0 \to M_0$ and
$E_1 \to M_1$ with Cartan connections $\omega_0$ and $\omega_1$
is an $H$-equivariant diffeomorphism 
$\map[F]{E_0}{E_1}$ 
so that $F^* \omega_1 = \omega_0$.
\end{definition}

\begin{definition}
A $G/H$-geometry is \emph{effective}
if $G$ acts faithfully on $G/H$.

On the other hand, suppose that
$H \subset G$ is a closed
subgroup and $K \subset H$
is the kernel of $G/H$. Let $\bar{G}=G/K$
and $\bar{H}=H/K$.
If $E \to M$
is a $G/H$-geometry, let $\bar{E}=E/K$.
Denote by $\LieK$ the Lie algebra of $K$.
There is a unique 1-form $\bar{\omega}$
on $\bar{E}$ which pulls back to 
$\omega+\LieK$; $\bar{E} \to M$ is 
the \emph{induced effective} 
$\bar{G}/\bar{H}$-geometry with
Cartan connection $\bar\omega$.
\end{definition}

\subsection{Curvature}

\begin{definition}
The \emph{curvature bundle} of a Cartan geometry $\map{E}{M}$
is the vector bundle 
\[
W =
E \times_H 
\left(
\LieG \otimes
\Lm{2}{\LieG/\LieH}^*
\right).
\]
\end{definition}

\begin{definition}
The \emph{curvature} of a Cartan geometry $\map{E}{M}$
with Cartan connection $\omega$ is 
the section of the curvature bundle
determined by the function 
$\map[K]{E}{W}$ for which
\[
d \omega + \frac{1}{2}\left[\omega,\omega\right] =
\frac{1}{2}K \omega \wedge \omega.
\]
\end{definition}

\begin{definition}
A Cartan geometry is \emph{flat}
if its curvature vanishes.
\end{definition}

\subsection{Flatness}

%

\begin{definition}
Let \(H \subset G\) be a closed subgroup of a Lie group.
Let \(X=G/H\) and \(x_0=1\cdot H \in X\).
Suppose \(\left(\delta,h\right)\) is
a developing system, with 
\[
\map[\delta]%
{\left(\tilde{M},\tilde{m}_0\right)}%
{\left(X,x_0\right)}
\]
the a developing map and
\[
\map[h]{\fundgp{M}}{G}
\]
the holonomy morphism.
Treat \(G\) as an \(H\)-bundle \(H \to G \to G/H=X\).
Then the \emph{associated flat Cartan geometry}
has bundle 
\(E = \fundgp{M} \backslash \delta^* G \to M\).
The bundle \(\map{E}{M}\) 
is a principal right \(H\)-subbundle
of the flat \(G\)-bundle \(\tilde{M} \times_h G \to M\).
The pullback 1-form \(g^{-1} \, dg\) is invariant under
this \(\fundgp{M}\)-action, so descends to 
a Cartan connection on \(E\).
\end{definition}

\begin{definition}
Suppose that $E \to M$ is a flat $G/H$-geometry.
Pick a point $e_0 \in E$. Let $m_0 \in M$
be the image of $e_0$. Suppose that
\[
\map{\left(\tilde{M},\tilde{m}_0\right)}%
{\left(M,m_0\right)}
\]
is the universal covering space.
Let $\tilde{E}$ be the pullback bundle
\[
\xymatrix{
\left(\tilde{E},\tilde{e}_0\right) \ar[d] \ar[r] &
\left(E,e_0\right) \ar[d] \\
\left(\tilde{M},\tilde{m}_0\right) \ar[r] & 
\left(M,m_0\right).
}
\]
Then $\tilde{E} \times G$ has 
foliation $\omega=g^{-1} \, dg$.

Let $S \subset \tilde{E} \times G$
be the smallest
$H$-invariant subset of $\tilde{E} \times G$
containing the leaf of the foliation
through $\left(\tilde{e}_0,1\right)$.
It is easy to prove that $S$ is the graph
of a unique $H$-equivariant local 
diffeomorphism
$\tilde{E} \to G$, say 
\[
\map[\Delta]{\tilde{E}}{G}
\]
so that $\Delta^* g^{-1} \, dg = \omega$.
We call $\Delta$ the \emph{bundle developing map}
associated to the choice of point $e_0 \in E$.
\end{definition}

\begin{definition}
Suppose that $E \to M$ is a flat $G/H$-geometry.
Let $\Delta$ be the bundle developing map
associated to $e_0 \in E$.
Set $\delta$ to be the quotient
of $\Delta$ by $H$.
By left invariance of the Pfaffian
system, for every $\gamma \in \fundgp{M}$,
there is some $g \in G$ so that
$\delta \circ \gamma=g \delta$.
We let $h\left(\gamma\right)=g$.
The pair $\left(\delta,h\right)$ 
is a developing system induced by the
flat Cartan geometry $\map{E}{M}$
and the choice of point $e_0 \in E$. 
\end{definition}

\begin{example}
Clearly every flat Cartan geometry is isomorphic 
to the associated flat Cartan geometry of a 
developing system (precisely the
developing system just defined).
If we vary the choice of point
$e_0$ by $H$-action, the developing
map varies by conjugacty.
So the category
of flat Cartan geometries $\map{E}{M}$
with chosen point $e_0 \in E$ modulo
isomorphism is $H$-equivariantly isomorphic
to the category of developing systems.
But the category of $G/H$-structures 
is only a quotient of the category of
developing systems, perhaps not
isomorphic when $G$ doesn't act
faithfully on $G/H$.
\end{example}

\subsection{Automorphism groups}

\begin{definition}
Let \(H \subset G\) be a closed subgroup of a Lie group.
An \emph{isomorphism} of $G/H$-geometries
$\map[\pi_0]{E_0}{M_0}$, with Cartan connection $\omega_0$,
and $\map[\pi_1]{E_1}{M_1}$, 
with Cartan connection $\omega_1$,
is a pair $\left(\varphi,\Phi\right)$ where
$\map[\varphi]{M_0}{M_1}$ is a diffeomorphism,
$\map[\Phi]{E_0}{E_1}$ is an $H$-equivariant
diffeomorphism so that $\Phi^* \omega_1 = \omega_0$
and $\pi_1 \circ \Phi = \varphi \circ \pi_0$.
\end{definition}

\begin{theorem}[Kobayashi \cite{Kobayashi:1995}]
The automorphism group of any Cartan geometry
\(\map{E}{M}\)
is a Lie group acting smoothly on $E$ and $M$.
Each orbit in $E$ is a smooth embedding of 
the automorphism group. 
\end{theorem}

\begin{definition}
An \emph{isomorphism} of $G/H$-structures 
is a diffeomorphism which pulls charts back to charts.
\end{definition}

\begin{definition}
Let \(H \subset G\) be a closed subgroup of a Lie group.
Suppose that $E \to M$ is a Cartan geometry.
Pick a point $e_0 \in E$ and let $m_0 \in M$
be its image. 
Let 
\[
\map{\left(\tilde{M},\tilde{m}_0\right)}%
{\left(M,m_0\right)}
\]
Define $\tilde{E}$ to be the pullback bundle
\[
\xymatrix{
\tilde{E} \ar[d] \ar[r] & E \ar[d] \\
\tilde{M} \ar[r] & M.
}
\]
Pullback $\omega$ to $\tilde{E}$.
Let $\fundgp{M}$ act on $\tilde{M} \times E$
by just acting on $\tilde{M}$.
An \emph{extended automorphism}
of the Cartan geometry with base point $e_0$
is an $H$-equivariant and 
$\fundgp{M}$-equivariant smooth map
\[
\map[\tilde\Phi]{\tilde{E}}{\tilde{E}},
\]
so that $\tilde\Phi^* \omega = \omega$.

If $\tilde\Phi$ is an extended automorphism,
define 
\[
\map[\tilde\varphi]{\tilde{M}}{\tilde{M}}
\]
to be the quotient by $H$-action,
define
\[
\map[\Phi]{E}{E}
\]
to be the quotient by $\fundgp{M}$-action,
and define
\[
\map[\varphi]{M}{M}
\]
to be the quotient by both actions.
\end{definition}

\begin{example}
Clearly if $\varphi$ and $\Phi$ are the identity,
then $\tilde\varphi$ can be any element of
the fundamental group of $M$, and 
$\tilde\Phi$ is completely determined.
\end{example}

\begin{remark}
A diffeomorphism 
$\map[\psi]{\tilde{M}}{\tilde{M}}$
is the lift $\psi=\tilde\varphi$
of a diffeomorphism 
$\map[\varphi]{M}{M}$
just when there is a group automorphism 
\[
\mapto{\gamma \in \fundgp{M}}{\gamma' \in \fundgp{M}}
\]
so that $\psi \circ \gamma = \gamma' \circ \psi$.
\end{remark}

\begin{example}
Every automorphism $\left(\varphi,\Phi\right)$
of a Cartan geometry lifts to some
extended automorphism, by taking 
$\tilde\varphi$ to be any lift to a map
on the universal covering space,
and then defining $\tilde\Phi$ by
\[
\begin{tikzpicture}[
back line/.style={},
cross line/.style={preaction={draw=white, -,
line width=6pt}}]
\matrix (m) [matrix of math nodes,
row sep=3em, column sep=3em,
text height=1.5ex,
text depth=0.25ex]{
& \tilde{E} & & E \\
\tilde{E}  & & E \\
& \tilde{M} & & M \\
\tilde{M} & & M \\
};
\path[->]
(m-1-2) edge (m-1-4)
edge  node[above] {$\tilde\Phi$} (m-2-1)
edge [back line] (m-3-2)
(m-1-4) edge (m-3-4)
edge node[above] {$\Phi$} (m-2-3)
(m-2-1) edge [cross line] (m-2-3)
edge (m-4-1)
(m-3-2) edge [back line] (m-3-4)
edge [back line] node[above] {$\tilde\varphi$} (m-4-1)
(m-4-1) edge (m-4-3)
(m-3-4) edge node[above] {$\varphi$} (m-4-3)
(m-2-3) edge [cross line] (m-4-3);
\end{tikzpicture}
\]
\end{example}

\begin{definition}
Let $\Aut{E}$ be the automorphism group,
and
\[
1 \to \fundgp{M,m_0} \to \ExtAut{E} \to \Aut{E} \to 1
\]
be the obvious exact sequence of the extended
automorphism group.
\end{definition}

\begin{definition}
Let \(H \subset G\) be a closed subgroup of a Lie group.
Suppose that $E \to M$ is a flat Cartan geometry.
Pick a point $e_0 \in E$ and let $m_0 \in M$
be its image. 
Pick $\tilde\Phi \in \ExtAut{E}$,
and as above define $\varphi,\tilde\varphi,\Phi$.
Pullback $\omega$ to $\tilde{E}$.
Suppose that
\[
\map[\Delta]{\tilde{E}}{G}
\]
is the developing map of $E$,
so that $\Delta^* g^{-1} \, dg = \omega$.
Define
\[
h\left(\tilde\Phi\right)
=
\Delta \circ \tilde\Phi \left(\tilde{e}_0\right)
\]
and call the map 
$\map[h]{\ExtAut{E}}{G}$ 
the \emph{extended holonomy morphism}.
\end{definition}

\begin{lemma}
The extended holonomy morphism of any
developing system on any connected
manifold is a Lie group
morphism, which agrees on $\fundgp{M}$
with the holonomy morphism.
\end{lemma}
\begin{proof}
It is sufficient to prove the result
on $\tilde{M}$, so we can assume
that $M$ is simply connected.
Therefore the extended automorphism
group is really just the automorphism
group.
Pick two automorphisms 
$\Phi_1, \Phi_2$
and let $\Phi=\Phi_2 \circ \Phi_1$.
Let 
\[
g_j = 
\Delta\left(\Phi_j\left(e_0\right)\right),
\]
and
\[
g = 
\Delta\left(\Phi\left(e_0\right)\right).
\]
We need to prove that
\[
g=g_2 g_1.
\]
Start by assuming that
$\Phi_1\left(e_0\right)$
and
$\Phi_2\left(e_0\right)$
lie in the same path component as $e_0$.
Take smooth paths $e_1(t), e_2(t) \in E$ so that
$e_1(0)=e_2(0)=e_0$ and 
$e_1(1)=\Phi_1\left(e_0\right)$
and
$e_2(1)=\Phi_2\left(e_0\right)$
and let $e(t)=e_1(t)$ for $0 \le t \le 1$
and then $e(t)=\Phi \circ e_2(t-1)$
for $1 \le t \le 2$.
Let $g_1(t)=\Delta\left(e_1\left(t\right)\right)$
and $g_2(t)=\Delta\left(e_2\left(t\right)\right)$
and $g(t)=\Delta\left(e(t)\right)$.
Then for $t \ge 1$, 
\[
g(t)^{-1} \frac{dg}{dt} = g_2(t-1)^{-1} 
\left.\frac{dg_2}{dt}\right|_{t-1},
\]
so
\[
g(t) = g_1(1) g_2(t-1).
\]
Let $t=2$.

If these $\Phi_1\left(e_0\right)$ and
$\Phi_2\left(e_0\right)$ do not lie in 
the same path component as $e_0$, then
we make use of $H$-equivariance to
get them to.
\end{proof}

\begin{lemma}\label{lemma:Extension}
Suppose that $H \subset G$ is a closed
subgroup of a Lie group.
Suppose that \(\left(\delta,h\right) \) is a 
developing system for a $G/H$-structure
on a manifold $M$.
Let \(\map{\tilde{M}}{M}\)
be the universal covering space,
Suppose that $L$ is a Lie group
acting faithfully on $\tilde{M}$. 
Suppose that the holonomy morphism
$h$ extends to a Lie group morphism 
$\map[h]{L}{G}$, and 
that the developing map $\delta$ satisfies
\[
\delta \circ \ell =  h\left(\ell\right) \delta,
\]
for any $\ell \in L$. Then there is 
a unique injective Lie group 
morphism 
\[
L \to \ExtAut{E}
\]
commuting with the action on $\tilde{M}$.
\end{lemma}
\begin{proof}
Since the Cartan geometry is 
flat, we can assume that the Cartan
geometry is constructed from the
developing map, so $\tilde{E}=\delta^*G$.
Each element of $\delta^*G$ is a pair
$\left(\tilde{m},g\right)$
for which $g^{-1} \delta\left(\tilde{m}\right)=x_0 \in X$.
We define an extended automorphism
\[
\tilde{\Phi}\left(\tilde{m},g\right)
=
\left(\ell\tilde{m},h\left(\ell\right)g\right).
\]
\end{proof}

\section{Classification on curves}

\begin{longtable}{AAAAA}
\caption{The classification of 1-dimensional
connected complex homogeneous spaces $G/H$ with
$G$ acting holomorphically, faithfully 
and transitively on $G/H$,
and of all connected
compact complex curves admitting
holomorphic $G/H$-structures,
and a description of the moduli space
of $G/H$-structures on each curve.
For the second row,
see example~\vref{example:G1G2G3}.
}%
\label{table:OneD} \\
\textbf{Homogeneous} & 
\textbf{Group} & 
\textbf{Stabilizer} &
\textbf{Curve} &
\textbf{Moduli} 
\\
\textbf{curve} & 
& 
\textbf{subgroup} &
& 
\textbf{space}
\\
\endfirsthead
\multicolumn{5}{c}%
{{\tablename\ \thetable{}: continued from previous page}} \\
\textbf{Homogeneous} & 
\textbf{Group} & 
\textbf{Stabilizer} &
\textbf{Curve} &
\textbf{Moduli} 
\\
\textbf{curve} & 
& 
\textbf{subgroup} &
& 
\textbf{space}
\\
\endhead
\endfoot
\endlastfoot
\C{} &
\C{} &
0 &
\C{}/\Lambda & 
\Cohom{0}{\kappa} \setminus 0 \cong \C{\times}
\betweenEntries
\C{} &
A \ltimes \C{} &
A &
\C{}/\Lambda & 
\text{see lemma~\ref{lemma:Astructures}}
\betweenEntries
\C{} &
\C{\times} \ltimes \C{} &
\C{\times} &
\C{}/\Lambda & 
\Cohom{0}{\kappa} \cong \C{}
\betweenEntries
\C{}/\Lambda_0 &
\C{}/\Lambda_0 &
0 &
\C{}/\Lambda_1 &
\Cohom{0}{\kappa} \cong \C{}
\betweenEntries
\C{}/\Lambda_0 &
\Z{}_2 \ltimes \C{}/\Lambda_0 &
\Z{}_2 &
\C{}/\Lambda_1 &
\Cohom{0}{\kappa^{\otimes 2}} \cong \C{}
\betweenEntries
\C{}/\Lambda_{\text{hex}} &
\Z{}_6 \ltimes \C{}/\Lambda_{\text{hex}} &
\Z{}_6 &
\C{}/\Lambda_1 &
\Cohom{0}{\kappa^{\otimes 6}} \cong \C{}
\betweenEntries
\C{}/\Lambda_{\text{hex}} &
\Z{}_3 \ltimes \C{}/\Lambda_{\text{hex}} &
\Z{}_3 &
\C{}/\Lambda_1 &
\Cohom{0}{\kappa^{\otimes 3}} \cong \C{}
\betweenEntries
\C{}/\Lambda_{\text{square}} &
\Z{}_4 \ltimes \C{}/\Lambda_{\text{square}} &
\Z{}_4 &
\C{}/\Lambda_1 &
\Cohom{0}{\kappa^{\otimes 4}} \cong \C{}
\betweenEntries
\C{\times} &
\C{\times} &
1 &
\C{}/\Lambda & 
\Cohom{0}{\kappa} \setminus 0 \cong \C{\times}
\betweenEntries
\C{\times} &
\Z{}_2 \ltimes \C{\times} &
\Z{}_2 &
\C{}/\Lambda & \Cohom{0}{\kappa^{\otimes 2}}
\setminus 0 \cong \C{\times} 
\betweenEntries
\Proj{1} &
\PSL{2,\C{}} &
\begin{bmatrix}
a & b \\
0 & \frac{1}{a}
\end{bmatrix}
& \C{}/\Lambda & \Cohom{0}{\kappa^{\otimes 2}} \cong \C{}
\betweenEntries
\Proj{1} &
\PSL{2,\C{}} &
\begin{bmatrix}
a & b \\
0 & \frac{1}{a}
\end{bmatrix}
& C_{g \ge 2} & \Cohom{0}{\kappa^{\otimes 2}} \cong \C{3g-3}
\betweenEntries
\Proj{1} & 
\PSL{2,\C{}} &
\begin{bmatrix}
a & b \\
0 & \frac{1}{a}
\end{bmatrix}
 & \Proj{1} & * \\ 
\end{longtable}

A holomorphic Cartan geometry on a complex
curve must be flat, because the curvature
bundle is the zero bundle. If
it is effective, then
it must arise from a unique locally
homogeneous geometric structure.
Table~\ref{table:OneD} gives a brief survey 
of the locally homogeneous structures on compact
complex curves, modelled on complex-homogeneous
curves with faithful
actions. We will now describe all such
structures explicitly.

\begin{example}\label{example:TranslationAndAffine}
Every elliptic curve $E=\C{}/\Lambda$ 
has an obvious translation structure, i.e. a 
$G/H$-structure where 
$G=\C{}$ and $H=0$.
The developing map is
\[
\mapto[\delta]{z \in \C{}}{cz \in \C{}}
\]
for any constant $c \ne 0$
and the holonomy morphism is
\[
\mapto[h]{\lambda \in \Lambda}{c\lambda \in \C{}}.
\]
A translation structure is nothing but
a holomorphic nowhere vanishing 1-form.
The automorphisms of any translation structure
are the translations.
\end{example}

\begin{example}
Every elliptic curve $E=\C{}/\Lambda$ has an obvious 
affine structure,
i.e. a $G/H$-structure where $G=\C{\times} \rtimes \C{}$, $H=\C{\times}$.
The developing map is the identity
\[
\mapto[\delta]{z \in \C{}}{z \in \C{}}
\]
and the holonomy morphism is
\[
\mapto[h]{\lambda \in \Lambda}{\left(1,\lambda\right) \in G}.
\]
There are more affine structures:
if the elliptic curve is $E=\C{}/\Lambda$,
then we can produce, for every $c \in \C{\times}$,
an affine structure on $E$ via the developing
map 
\[
\mapto[\delta]{z \in \C{}}{e^{cz} \in \C{\times}},
\]
with holonomy morphism
\[
\mapto[h]{\lambda \in \Lambda}{\left(e^{c\lambda},0\right) \in G}.
\]
Since the developing map of an affine
structure is unique precisely up to affine
transformations, we could just as well use the
map
\[
\mapto{z}{\frac{e^{cz}-1}{c}},
\]
and think of \mapto{z}{z} as the case $c=0$.
To forget the choice of affine coordinate
$z$, we can identify the moduli space of affine structures
on an elliptic curve $E$ with 
the space of holomorphic 1-forms on $E$:
each 1-form is expressed
in terms of an affine chart as $c \, dz$.
It is well known \cite{Loray/MarinPerez:2009}
that every holomorphic affine structure
arises uniquely in this way.
The automorphism group of the affine
structure with linear developing map
is the biholomorphism group of the elliptic curve $C$.
The automorphism group of any of the affine
structures with nonlinear developing map
is the group of translations.
\end{example}

\end{unabridged}

\begin{abridged}

\section{Example}

\end{abridged}

\begin{example}\label{example:G1G2G3}
%
Pick any
countable abelian group $A \subset \C{\times}$
with the discrete topology.
Set $G = A \ltimes \C{}$ and $H=A$.
The group $G$ acts holomorphically,
faithfully and transitively on 
$X=\C{}=G/H$.
Every complex Lie group
$G$ acting holomorphically,
faithfully and transitively
on $\C{}$ and not equal to the
affine group of $\C{}$ has this
form. The group $A$ need not
be finitely generated; for
example $A$ could be the group
of all roots of 1.
Every translation
structure on any elliptic curve
is a $G/H$-structure.

Let 
\[
A'
=
\set{\alpha \in \C{}|e^{\alpha} \in A},
\]
giving the exact sequence of
abelian groups 
$0 \to \Z{} \to A' \to A \to 0$.
Pick any elliptic curve $C=\C{}/\Lambda$.
A \emph{grain} for $\Lambda$ and $A$
is a complex number $c \ne 0$ so that 
$c \Lambda \subset A'$.
(If $A' \subset \C{}$ is a lattice, then
a grain is a covering map
$C \to \C{}/A'$.)
Suppose that $c$ is a grain.
Then for any constant $k \ne 0$, the developing map
\[
\mapto[\delta]{z \in \C{}}{ke^{cz} \in \C{}}
\]
and holonomy morphism
\[
\mapto[h]{\lambda \in \Lambda}%
{\left(e^{c \lambda},0\right) \in G}
\]
gives a $G/H$-structure. 
Let $\Gamma\left(\Lambda,A\right)$
be the set of grains. {}%
\begin{abridged}%
Every $G/H$-structure is a translation
structure or constructed from a grain.
The moduli
space of $G/H$-structures is
\[
\left(
\Cohom{0}{\kappa}/A 
\right)
\cup 
\bigcup_{c \in \Gamma\left(\Lambda,A\right)}
\C{\times}/A.
\]
\end{abridged}

The automorphisms of
a $G/H$-structure
induced from a translation structure
are the automorphisms
of $\C{}/\Lambda$, say $z \mapsto az+b$,
for which $a \in A$ and $a \Lambda=\Lambda$.
So the automorphism group is 
$\Z{}_n \ltimes \C{}$ for some
integer $n=1,2,3,4$ or $6$.

The automorphisms of a $G/H$-structure
with grain $c \ne 0$
are the translations $z \in \C{}/\Lambda$
for which $e^{cz} \in A$, a finite group.
In particular, these $G/H$-structures are
not homogeneous.
\end{example}

\begin{unabridged}

\begin{lemma}\label{lemma:Astructures}
Take a countable abelian Lie group $A \subset \C{\times}$
with the discrete topology.
Let $G = A \ltimes \C{}$, $H=A$ and let $X=G/H=\C{}$.
Suppose that $C$ is a compact complex curve
bearing a $G/H$-structure. Then that
$G/H$-structure is one of the examples
in example~\vref{example:G1G2G3}, up
to conjugacy. The moduli space of
$G/H$-structures induced by
translation structures is 
$\left(\Cohom{0}{\kappa} \setminus 0\right)/A$,
which might not be Hausdorff.
The moduli
space of $G/H$-structures is
\[
\left(
\Cohom{0}{\kappa}/A 
\right)
\cup 
\bigcup_{c \in \Gamma\left(\Lambda,A\right)}
\C{\times}/A.
\]
\end{lemma}
\begin{remark}
The topology on this moduli space is not clear.
However, it is clear that there are choices
of $\Lambda$ and $A$ for which the moduli
space is not Hausdorff. 
For example, pick any
$\alpha \in \C{\times}$ 
with $\left|\alpha\right|>1$
and let $A \subset \C{\times}$ 
be the group generated
by $\alpha$ together with all roots of unity
and all nonzero rational numbers.
Let $\Lambda=\left<2\pi i, \log \alpha\right>
\subset \C{}$. 
The moduli stack
of $\left(A \ltimes \C{}\right)/A$-structures
on $\C{}/\Lambda$ 
is not representable by a complex analytic space,
or even by a complex analytic orbispace.
\end{remark}
\begin{proof}
Since $G$ is a subgroup
of the affine group, the developing
map must be the same, up to affine
transformation, as the developing
map of some affine structure. We can conjugate
to eliminate any translation.
So either
a translation structure or
\begin{align*}
\delta(z) &= ke^{cz}, \text{ and } \\
h(\lambda) &= \left(e^{c\lambda},0\right)
\end{align*}
But $h$ is valued in $G$, so $c$ is a grain.
\end{proof}

\begin{example}
A \emph{projective structure} is a 
$G/H$-structure where $G=\PSL{2,\C{}}$ and $H$
is the Borel subgroup of matrices of the form
\[
\begin{bmatrix}
a & b \\
0 & \frac{1}{a}
\end{bmatrix}.
\]
Every affine structure imposes a projective structure
via the obvious inclusion $\C{}=\Proj{1} \setminus \infty \to \Proj{1}$.
Among the various affine structures on elliptic curves
given above, each induces a distinct projective structure,
except for pairs $c$ and $-c$,
which are clearly related by a linear fractional
transformation 
\[
e^{-cz} = \frac{1}{e^{cz}}.
\]
It turns out that these are all of the projective structures:
we can identify the moduli space of
projective structures on any elliptic curve
with the space of quadratic
differentials $c^2 \, dz^2$.
For $c \ne 0$, the developing map is
\[
\mapto[\delta]{z \in \C{}}{e^{cz} \in \Proj{1}}
\]
and the holonomy morphism is
\[
\mapto[h]{\lambda \in \Lambda}%
{
\begin{bmatrix} 
e^{c\lambda/2} & 0\\
0 & e^{-c\lambda/2}
\end{bmatrix}
\in G}.
\]
For $c=0$, the developing map is
\[
\mapto[\delta]{z \in \C{}}{z \in \Proj{1}}
\]
and the holonomy morphism is
\[
\mapto[h]{\lambda \in \Lambda}%
{
\begin{bmatrix} 
1 & \lambda \\
0 & 1
\end{bmatrix}
\in G}.
\]

The automorphism group of the projective
structure at $c=0$ is the biholomorphism
group of the elliptic curve.
At $c \ne 0$, the automorphism
group is $\Z{}_2 \ltimes \C{}$,
where $\Z{}_2$ acts on $\C{}/\Lambda$
by $z \mapsto -z$, and $\C{}$ acts
by translations.
\end{example}

\begin{example}
Clearly $\Proj{1}$ has an obvious projective
structure
The developing map is
\[
\mapto[\delta]{z \in \Proj{1}}{z \in \Proj{1}}
\]
and the holonomy morphism is
\[
\mapto[h]{1}%
{
\begin{bmatrix} 
1 & 0\\
0 & 1
\end{bmatrix}
\in G}.
\]
The biholomorphism group is $G=\PSL{2,\C{}}$.
This is the only effective 
holomorphic Cartan geometry on $\Proj{1}$.
\end{example}

\begin{example}[Cartan \cite{Cartan:1992} part 3 p. 3]%
\label{example:ProjConnOnCurve}
Suppose that $C$ is a Riemann surface.
Suppose that $\eta$ is a quadratic
differential on $C$, say $\eta = f(z) \, dz$
in some local coordinate $z$. Let $F$ 
be any solution of
\[
\left\{F,z\right\}=f,
\]
where $\left\{F,z\right\}$ indicates
the Schwarzian derivative.
Use $F$ as a local biholomorphism
of an open subset of $C$ to $\Proj{1}$.
Since each solution $F$ is unique up to linear
fractional transformation, the
collection of all solutions 
forms a $\Proj{1}$-structure.
A developing map is precisely
a global solution $\delta=F$
on the universal covering space of $C$.
There is no elementary expression
for the holonomy; see 
Loray and  Mar{\'{\i}}n P{\'e}rez
\cite{Loray/MarinPerez:2009}.
The automorphism group is the finite 
group of biholomorphisms preserving
the quadratic differential.
\end{example}

\begin{example}\label{example:TorusTorus}
Suppose that $\Lambda_0, \Lambda_1 \subset V$
are lattices in a finite dimensional complex vector space $V$.
(We are particularly interested in $V=\C{}$.)
Let $T_j = V/\Lambda_j$. Then $T_1$ has a $T_0$-structure, 
i.e. a $G/H$-structure where $G=T_0$ and $H=0$,
with developing map
\[
\mapto[\delta]{z \in V}{z + \Lambda_0 \in T_0}
\]
and holonomy morphism
\[
\mapto[h]{\lambda \in \Lambda_1}{\lambda + \Lambda_0 \in T_0}.
\]
The automorphism group is the group of translations
together with all those linear isomorphisms
of $V$ which preserve both lattices.
\end{example}

\begin{lemma}\label{lemma:torusTorusTorus}
Suppose that $\Lambda_0 \subset V_0$ and 
$\Lambda_1 \subset V_1$
are lattices in finite dimensional complex 
vector spaces $V_0$
and $V_1$ of equal dimension. Let $T_j=V_j/\Lambda_j$. 
Up to isomorphism, the only $T_0$-structure on $T_1$ 
is obtained by taking a linear isomorphism 
\map[L]{V_1}{V_0} that makes $\Lambda_0$
a sublattice of $\Lambda_1$ and then following
the recipe of example~\vref{example:TorusTorus}.
\end{lemma}
\begin{proof}
Suppose that we have a $T_0$-structure on $T_1$, with
developing map
\[
\map[\delta]{V_1}{T_0}
\]
and holonomy morphism
\[
\map[h]{\Lambda_1}{T_0}.
\]
By translation, we can assume that $\delta(0)=0$.
Since $V$ is simply connected, $\delta$ lifts uniquely
to a map
\[
\map[\tilde{\delta}]{V_1}{V_0}
\]
so that $\tilde{\delta}(0)=0$. Define
\[
\mapto[\tilde{h}]{\lambda \in \Lambda_1}%
{\tilde{\delta}\left(\lambda\right) \in V_0}.
\]
The function
\[
\tilde{\delta}\left(z+\lambda\right)-\tilde{\delta}\left(z\right)
\]
equals $\tilde{h}\left(\lambda\right)$ at $z=0$, 
and has derivatives the same as
\[
\delta\left(z+\lambda\right)-\delta\left(z\right)
\]
i.e. vanishing, so is constant. Therefore 
$\tilde{\delta}$ is the developing map, 
and $\tilde{h}$ the holonomy morphism,
of a unique translation structure on $T_1$,
i.e.
\[
\tilde{\delta}\left(z+\lambda\right)-\tilde{\delta}\left(z\right)
=
\tilde{h}\left(\lambda\right).
\]
Therefore the exterior derivative is
\[
d\tilde{\delta}_{z+\lambda} = d\tilde{\delta}_z,
\]
a periodic map $V_1 \to V_1^* \otimes V_0$, so
constant. Therefore $\tilde{\delta}$ is affine,
and so a linear isomorphism, say $L=\tilde{\delta}$.
We can therefore assume that $V_0=V_1$ and $L$ is the identity map.
Clearly 
$h\left(\lambda\right)=\delta\left(\lambda\right)-\delta(0)=
\lambda+\Lambda_0 \in T_0$.
\end{proof}

\begin{lemma}\label{lemma:torusTorus}
Suppose that $E_0$ is an elliptic curve. 
The only compact complex curves which admit
$E_0/0$-structures are elliptic curves.
The $E_0/0$-structures on an elliptic curve
$E_1$ are parameterized
by the linear isomorphisms between the Lie 
algebras of translations of $E_0$ and $E_1$.
\end{lemma}
\begin{proof}
Suppose that $C$ is a compact complex
curve equipped with a holomorphic $E_0$-structure.
Since $E_0$ has an
invariant affine structure, so does $C$,
and therefore $C$ is an elliptic curve.
By lemma~\vref{lemma:torusTorusTorus},
the $E_0/0$-structures on $C$ must be
as described.
\end{proof}

\begin{example}\label{example:Zn}
Take an elliptic curve $E_0=\C{}/\Lambda_0$.
Suppose that $\Lambda_0 \subset \C{}$
is invariant under the group $\Z{}_n$
of multiplication
by $n$-th roots of 1. (For example,
any $\Lambda_0$ is invariant under
$\Z{}_2$.)
Take $G=\Z{}_n \ltimes \C{}/\Lambda_0$,
and $H=\Z{}_n$. Any $G/H$-structure imposes
an affine structure, since $G$ preserves
the affine structure on $E_0$.
Therefore the only compact complex 
curves which admit $G/H$-structures are 
elliptic curves. 
Take a primitive $n$-th root of 1, say $\alpha$.
Write elements of $G$ as pairs
$(a,b)$, with $a=\alpha^k$,
acting on $z \in E_0$
by $\mapto{z}{az+b}$.
If $E_1=\C{}/\Lambda_1$
is an elliptic curve, for any complex
number $c \in \C{\times}$, we can take 
developing map
\[
\mapto[\delta]%
{z \in \C{}}%
{cz + \Lambda_0 \in E_0},
\]
and holonomy morphism
\[
\mapto[h]%
{\lambda \in \Lambda_1}%
{(1,c\lambda) \in G}.%
\]
It is easy to check
that the constant $c$ is
uniquely determined up to $c \cong \alpha^k c$,
under conjugation, for $k=0, 1, \dots, n-1$. 
We want to prove
that $\left(c \, dz\right)^n$ 
parameterizes the moduli
space.

The automorphism group of the 
structure is the group of automorphisms
$z \mapsto \alpha^k z + b$,
where $\alpha^k$
preserves the lattice $\Lambda$,
and $b \in \C{}$.
\end{example}

\begin{lemma}\label{lemma:Zn}
Take an elliptic curve $E_0=\C{}/\Lambda_0$,
where $\Lambda_0 \subset \C{}$
is a lattice invariant under the group $\Z{}_n$
of multiplication by $n$-th roots of 1.  
Take $G=\Z{}_n \ltimes \C{}/\Lambda_0$,
and $H=\Z{}_n$. Then up to conjugation,
every $G/H$-structure on any elliptic
curve is obtained as in example~\vref{example:Zn}.
In particular, the choice of element
$\left(c \, dz\right)^n$ parameterizes
the moduli space.
\end{lemma}
\begin{proof}
Suppose that
\[
\map[\delta]{\C{}}{E_0}
\]
is some developing map with holonomy
morphism
\[
\map[h]{\Lambda_1}{G},
\]
for a $G/H$-structure on an elliptic
curve $E_1=\C{}/\Lambda_1$.
So
\[
\delta\left(z+\lambda\right)
=
h\left(\lambda\right)\delta(z),
\]
for $z \in \C{}$ and $\lambda \in \Lambda_1$.
Write $h(\lambda)$ as
\[
z \mapsto a_{\lambda}z + b_{\lambda}.
\]
So
\[
\left(\delta'\left(z+\lambda\right)\right)^n
=
\delta'(z)^n,
\]
i.e. $\delta'(z)^n$ is an invariant holomorphic
function on $E_1$, and therefore constant.
Therefore $\delta'(z)$ is a nonzero constant, so
up to conjugacy
\[
\delta(z) = cz + \Lambda_0,
\]
for some constant $c \ne 0$. 
The relation
\[
\delta\left(z+\lambda\right)
=
h\left(\lambda\right)
\delta(z)
\]
then says that
\[
a_{\lambda}=1
\]
for all $\lambda \in \Lambda_1$,
and that 
\[
b_{\lambda}=c\lambda+\Lambda_0.
\]
\end{proof}

\begin{example}\label{example:Ctimes}
Consider $\C{\times}$-structures, i.e. $G/H$-structures 
where $G=\C{\times}$ and $H=\{1\}$.
Pick an elliptic curve $E=\C{}/\Lambda$.
The affine structure on $E$ has the developing
map 
\[
\mapto[\delta]{z \in \C{}}{e^{cz} \in \C{\times}},
\]
with holonomy morphism
\[
\mapto[h]{\lambda \in \Lambda}{e^{c\lambda} \in G},
\]
so is also a $\C{\times}$-structure. The automorphism
group is the group of translations.
\end{example}

\begin{lemma}\label{lemma:Cstar}
Up to conjugacy, the only $\C{\times}$-structures on compact
complex curves are those of example~\vref{example:Ctimes}.
\end{lemma}
\begin{proof}
Suppose that $C$ is a compact complex curve with
a $\C{\times}$-structure.
Clearly $\C{\times}$ acts on $\C{\times}$ preserving an 
affine structure, so $C$ has an affine structure,
and so $C$ is an elliptic curve, say $C=\C{}/\Lambda$.
Pick a developing map
\[
\map[\delta]{\C{}}{\C{\times}}
\]
with holonomy morphism
\[
\map[h]{\Lambda}{\C{\times}}.
\]
Since $\C{}$ is simply connected,
and $\delta$ has no zeroes, we can
let
\[
D(z) = \int \frac{\delta'(z)}{\delta(z)} \, dz.
\]
Then
\[
e^{D(z+\lambda)}=h(\lambda)e^{D(z)},
\]
so that
\[
e^{D\left(z+\lambda\right)-D(z)}
\]
is independent of $z$, so 
\[
D'\left(z+\lambda\right)=D'(z)
\]
is a holomorphic function of $z$ on $\C{}/\Lambda$,
so constant, and therefore $D'(z)$ is constant,
say $D(z)=az+b$, so $\delta(z)=Ce^{az}$
for some constant $C \ne 0$.
By conjugation, we can arrange that $C=1$.
Then $h(\lambda)=e^{a \lambda}$.
\end{proof}

\begin{example}\label{example:CtimesCtimes}
Let $G = \Z{}_2 \ltimes \C{\times}$
and $H = \Z{}_2$, where $\left(\pm, a\right) \in G$ acts on 
$X=\C{\times}$ by 
\begin{align*}
(+,a)z&=az, \\
(-,a)z&=\frac{a}{z}.
\end{align*}
We have the same examples we had
in example~\vref{example:Ctimes}: $\delta(z)=e^{cz}$
for $c \ne 0$. The automorphism group of each such
structure is $\Z{}_2 \ltimes \C{}$,
with $1 \in \Z{}_2$ acting by $z \mapsto -z$
on $C=\C{}/\Lambda$ and by $z \mapsto 1/z$ 
on $\Proj{1}$.
\end{example}

\begin{lemma}\label{lemma:CstarCstar}
The only $\Z{}_2 \ltimes \C{\times}/\Z{}_2$-structures 
on compact
complex curves are those of
example~\vref{example:CtimesCtimes}.
\end{lemma}
\begin{proof}
Suppose that a compact complex curve
$C$ had a $G/H$-structure.
Then at most a 2-1 covering space of $C$
would have a $\C{\times}/1$-structure.
So $C$ must be an elliptic curve, say
$C=\C{}/\Lambda$.
Suppose that the $(G,X)$-structure has developing
map 
\[
\map[\delta]{\C{}}{\C{\times}},
\]
and holonomy morphism
\[
\map[h]{\Lambda}{\C{\times} \sqcup \C{\times}}.
\]
Again, on a 2-1 covering space this 
$G/H$-structure becomes a $\C{\times}/1$-structure,
with the same developing map. Therefore up
to conjugacy the developing map must be
\[
\mapto[\delta]{z \in \C{}}{e^{cz} \in \C{\times}}.
\]
The holonomy morphism must be
\[
h(\lambda) = 
\frac{\delta(z+\lambda)}{\delta(z)}=e^{c\lambda}.
\]
Conjugation can then identify $c \cong -c$.
\end{proof}

\begin{theorem}
Every effective holomorphic locally homogeneous geometric
structure on a compact Riemann surface is
one of the examples%
~\ref{example:TranslationAndAffine}--\ref{example:Ctimes},
as described in table~\vref{table:OneD}.
\end{theorem}
\begin{proof}
Gunning \cite{Gunning:1967} proves the result
assuming that $G/H$ is simply connected.
Loray and  Mar{\'{\i}}n P{\'e}rez
\cite{Loray/MarinPerez:2009}
provide a clear exposition of Gunning's proof.

We let the reader prove that the biholomorphism
group of $\C{\times}$ is $\C{\times} \sqcup \C{\times}$,
via $z \mapsto cz$ for $c \ne 0$ and via $z \mapsto c/z$
for $c \ne 0$. Therefore the only
complex Lie groups acting faithfully on $\C{\times}$
are $\C{\times}$ and $\C{\times} \sqcup \C{\times}$.

The complete classification of complex homogeneous
curves is then clearly as given in table~\vref{table:OneD},
and lemmas~\ref{lemma:torusTorus}, 
\ref{lemma:Zn},
\ref{lemma:Cstar}
and \ref{lemma:CstarCstar} finish
the classification of compact complex curves
with locally homogeneous structures. The proofs
of these lemmas also provide parameterizations
of the moduli spaces.
\end{proof}

%
%

\subsection{Automorphism groups}

\begin{theorem}\label{theorem:homogeneity}
On a compact complex curve of genus
$0$ or $1$, every holomorphic locally
homogeneous structure 
is homogeneous, except for a 
$G/H$-structure with nonlinear
developing map associated to a 
grain; such a structure has finite automorphism
group. On a curve of 
higher genus, every holomorphic locally
homogeneous structure is a projective
structure with a finite automorphism
group. 
\end{theorem}
\begin{proof}
In each case of an elliptic curve,
the developing map was linear or
exponential, and the holonomy the
same linear or exponential restricted
to the lattice, so the holonomy morphism
extends to a morphism on the group of translations
of $\C{}$. This morphism has image in $G$,
except when the structure is constructed
from a grain.
For a projective line,
the only structure is the standard
projective structure. Any higher
genus curve has finite biholomorphism
group.
\end{proof}

\section{Biholomorphism groups of 
products of curves}

Our goal in this section is to classify
all of the complex Lie groups acting
holomorphically, faithfully and
transitively on complex surfaces of
the form $C \times \Proj{1}$, where
$C$ is a (not necessarily compact)
complex curve.

\begin{lemma}\label{lemma:P1xP1}
\[
\Bihol{\Proj{1} \times \Proj{1}}
=
\Z{}_2 \ltimes 
\left(
\PSL{2,\C{}} 
\times
\PSL{2,\C{}} 
\right)
\]
\end{lemma}
\begin{proof}
\[
G= \Z{}_2 \ltimes 
\left(
\PSL{2,\C{}} 
\times
\PSL{2,\C{}} 
\right)
\]
act on $S$ by having $\Z{}_2$ interchange
the factors and having the obvious
linear fractional transformations on
each factor.

Clearly $S$ 
has two obvious fibrations by rational
curves, $\Proj{1} \times *$ and $* \times \Proj{1}$.
To any compact 
complex curve $C$ in $S$ we associate
its intersection numbers $a$ and $b$ with any fiber of
each of these fibrations. If we have two
rational curves $C_1$ and $C_2$, with
intersection numbers with the fibrations
$a_1, b_1$ and $a_2, b_2$ respectively,
then their intersection
number is
\[
C_1 C_2 = a_1 b_2 + a_2 b_1,
\]
so vanishes if and only if both $C_1$ and $C_2$
are fibers of the same one of our two fibrations.
Therefore every biholomorphism of $S$
preserves the pair of fibrations. Every
biholomorphism therefore, after multiplication
by a suitable element of $G$,
preserves both fibrations, and acts trivially
on the base of each, so is the identity.
\end{proof}

\begin{lemma}\label{lemma:BiholEP1}
Suppose that $C$ is a compact
complex curve, and that $S=C \times \Proj{1}$
is a homogeneous complex surface.
Then $C$ is complex-homogeneous.
Either $C$ is isomorphic to $\Proj{1}$ or
$C$ is an elliptic
curve and the biholomorphism group
of $S$ is
\[
\Bihol{S}
=
\Bihol{C}
\times \PSL{2,\C{}}.
\]
\end{lemma}
\begin{proof}
Suppose that $S=C \times \Proj{1}$
and that $C$ is not isomorphic
to $\Proj{1}$.
So $C$ admits no nonconstant maps
$\Proj{1} \to C$. Every
rational curve in $S$ lies in a fiber
of the obvious fibration $S \to C$. 
So this fibration is biholomorphism
invariant. Since the biholomorphism
group of $S$ is transitive,
it acts transitively on the base $C$
of the fibration. Therefore $C$
is an elliptic curve. The group 
$G=\Bihol{C}
\times \PSL{2,\C{}}$ also acts 
transitively on the base. 
Any biholomorphism, after suitable
multiplication by an element of
$G$, becomes trivial on the base,
and so has the form of a 
holomorphic map
\[
C \to \PSL{2,\C{}}.
\]
On a 2-1 covering space $\hat{C} \to C$, this
lifts to a holomorphic map
\[
\hat{C} \to \SL{2,\C{}},
\]
which must be constant because $\hat{C}$
is compact and $\SL{2,\C{}}$ is affine.
So the biholomorphism
is an element of $\PSL{2,\C{}}$.
\end{proof}

\begin{example}
The biholomorphism group of
$\C{} \times \Proj{1}$ is infinite
dimensional, containing all maps of
the form
\[
\mapto{\left(z,w\right)}{\left(z,g(z)w\right)},
\]
where $\map[g]{\C{}}{\PSL{2,\C{}}}$
is any holomorphic map.
For example, we could take 
\[
g(z)
=
\begin{pmatrix}
e^{a(z)} & b(z) \\
0 & e^{-a(z)}
\end{pmatrix},
\]
for any two entire functions $a(z), b(z)$.
\end{example}

\begin{example}\label{example:ZZn}
There are some obvious finite dimensional
complex Lie groups of biholomorphisms
acting holomorphically, faithfully and transitively
on $\Proj{1} \times \C{}$.
One is
\[
\PSL{2,\C{}} \times \left(\C{\times} \ltimes \C{}\right),
\]
acting as $\Bihol{\Proj{1}} \times \Bihol{\C{}}$.

More generally, if $A \subset \C{\times}$ is a countable
abelian group with the discrete topology then let
\[
G = \PSL{2,\C{}} \times \left(A \ltimes \C{}\right),
\]
acting on $\Proj{1} \times \C{}$ as
\[
(g,a,t)(z,w)=\left(gz,aw+t\right).
\]
\end{example}

Next we will write down a series of examples of 
connected complex Lie groups acting faithfully
and transitively on connected complex surfaces.
We will then explain why this series of examples,
together with our previous examples, is
complete. We will then make use of this list
to find all of the (not necessarily connected)
complex Lie groups acting holomorphically, faithfully and
transitively on surfaces of the form
$C \times \Proj{1}$.

\begin{example}\label{example:FirstLie}
$G=\SL{2,\C{}}$ has the obvious action
on $X=\C{2} \setminus 0$.
\end{example}

\begin{example}
$G=\PSL{2,\C{}}$ acts on 
$\Proj{1} \times \Proj{1} \setminus \text{diagonal}$ 
by the usual linear fractional action
on each $\Proj{1}$.
\end{example}

\begin{example}
$G=\GL{2,\C{}}$ has the obvious action on $X=\C{2} \setminus 0$.
\end{example}

\begin{example}
Let $D$ be an effective divisor on $\C{}$,
and let $p(z)$ be the monic polynomial
with divisor $D$. Let $V=V_D$ be the
set of entire functions $\map[f]{\C{}}{\C{}}$
for which 
\[
p\left(\pd{}{z}\right)f=0.
\]
Let $G=G_D=\C{} \ltimes V_D$ with the group operation
\[
\left(t_0,f_0(z)\right)\left(t_1,f_1(z)\right)
=
\left(t_0+t_1,f_0(z)+f_1\left(z-t_0\right)\right).
\] 
Let $G$ act on $X=\C{2}$ by
\[
\left(t,f\right)\left(z,w\right)
=
\left(z+t,w+f\left(z+t\right)\right).
\] 
\end{example}

\begin{example}
Let $D$ be an effective divisor on $\C{}$,
and let $p(z)$ be the monic polynomial
with divisor $D$. Let $V=V_D$ be the
set of entire functions $\map[f]{\C{}}{\C{}}$
for which 
\[
p\left(\pd{}{z}\right)f=0.
\]
Let $G=G'_D = \C{} \times \C{\times} \times V_D$
with group operation
\[
\left(t_0,\lambda_0,f_0(z)\right)
\left(t_1,\lambda_1,f_1(z)\right)
=
\left(t_0+t_1,\lambda_0 \lambda_1, 
f_0(z) + \lambda_0 \, f_1\left(z-t_0\right)\right)
\]
and action on $X=\C{2}$
\[
\left(t,\lambda,f\right)(z,w)
=
\left(z + t, \lambda w + f\left(z + t\right)\right).
\]
\end{example}

\begin{example}\label{example:On}
The surface $\OO{n}$ is the total space of the usual line
bundle $\OO{n} = \OO{1}^{\otimes n} \to \Proj{1}$,
whose fibers are choices of line in $\C{2}$ and homogeneous 
polynomial of degree $n$ on that line.
This surface is acted on by the group $\GL{2,\C{}}$ of linear substitutions of variables, and is also acted on by the group $\Sym{n}{\C{2}}^*$
by adding a globally defined homogeneous polynomial to the polynomial on any given line.
The subgroup  $\Z{}_n \subset \GL{2,\C{}}$ of scalings of variables by $n$-th roots of unity acts trivially. 
So the complex surface $X=\OO{n}$
is acted on by
\[
G
= 
\left(\GL{2,\C{}}/\Z{}_n\right)
\ltimes 
\Sym{n}{\C{2}}^*.
\]
\end{example}

\begin{example}
The surface $X=\OO{n}$ is also acted on 
by
\[
G
= 
\left(\SL{2,\C{}}/Z\right)
\ltimes 
\Sym{n}{\C{2}}^*,
\]
where $Z$ is $\pm 1$ if $n$ is even
and is $1$ is $n$ is odd.
\end{example}

\begin{example}\label{example:C2On}
Take any global section $s$
of $\OO{n} \to \Proj{1}$
which vanishes precisely at 
$\infty \in \Proj{1}$.
Then over $\C{} \subset \Proj{1}$,
we trivialize $\OO{n}$ by the embedding
\[
\mapto{\left(z,w\right) \in X=\C{2}}%
{w \, s(z) \to \OO{n}}.
\] 
The subgroup of $G_0 \subset \GL{2,\C{}}$ of 
matrices of the form
\[
\begin{pmatrix}
a & b \\
0 & c
\end{pmatrix}
\]
acts on this open subset of $\OO{n}$ by
\[
\mapto{\left(z,w\right)}%
{\left(\frac{a}{d}z+\frac{b}{d},\frac{w}{d^n}\right)}.
\]
Then
\[
G = G_0 \ltimes \Sym{2}{\C{n}}^*,
\]
acts on $X=\C{2}$.
\end{example}

\begin{example}
Pick any complex number $\alpha$.
In the last example, we can replace $G_0$ by
the group of matrices of the form
\[
\mu
\begin{pmatrix}
e^{\lambda\left(1-\frac{\alpha}{n}\right)} & b \\
0 & e^{-\lambda \alpha/n}
\end{pmatrix}
\]
where $\mu$ is any $n$-th root of 1 and
$\lambda$ is any complex number.
\end{example}

\begin{example}
Consider the set $G$ of elements $(g,p)$ of
\[
\left(\GL{2,\C{}}/\Z{}_n\right)
\ltimes 
\Sym{n}{\C{2}}^*
\]
where 
\[
g=
\begin{pmatrix}
1 & b \\
0 & e^{-\lambda}
\end{pmatrix}
\]
and 
\[
p\left(z_1,z_2\right)
=
\lambda z_1^n + z_2 \, r\left(z_1,z_2\right),
\]
where
\[
r \in \Sym{n-1}{\C{2}}^*.
\]
This $G$ acts on the subset
$X=\C{2} \subset \OO{n}$
which we constructed
in example~\vref{example:C2On}.
\end{example}

\begin{example}
The affine group
\[
\GL{2,\C{}} \ltimes \C{2}
\]
and its subgroup
\[
\SL{2,\C{}} \ltimes \C{2}
\]
both act faithfully and transitively
on $\C{2}$.
\end{example}

\begin{example}\label{example:LastLie}
The group $\PSL{3,\C{}}$ acts
faithfully and transitively on $\Proj{2}$.
\end{example}

\begin{theorem}[Lie \cite{Mostow:1950,Olver:1995}]%
\label{theorem:Lie}
Suppose that $S=G/H$ is a simply connected
complex-homogeneous
complex surface, with $G$ a connected
complex Lie group
acting holomorphically, faithfully and transitively on $S$. 
Then either (1)
$S=\Proj{1} \times \Proj{1}$
and 
$G=\PSL{2,\C{}} 
\times
\PSL{2,\C{}}$ or
(2) 
$S$ is one of the examples~\ref{example:FirstLie}
through \ref{example:LastLie}.
\end{theorem}
\begin{proof}
Lie actually only identifies the possible
Lie algebras $\LieG$ and stabilizer subalgebras
$\LieH$. But then each such pair $\left(\LieG,\LieH\right)$
acts by a complete action (i.e. extending to a Lie group
action) on at most one simply connected
surface $S$, i.e. $S=G/H$ where $G$ is the 
connected and simply connected Lie group with Lie
algebra $\LieG$, and $H \subset G$ the connected
subgroup with Lie algebra $\LieH$. Mostow
\cite{Mostow:1950} proves that $H$ must actually
be a closed subgroup as long as $\LieH \subset \LieG$
has complex codimension 2. Since we have written
one simply connected example surface for
each of Lie's Lie algebra pairs, our list
must be exhaustive. 
\end{proof}

This classification of Lie
will not suffice; we need to know all
of the complex Lie groups $G$ (not necessarily
connected) acting 
holomorphically, faithfully and transitively
on connected 
complex surfaces of the form $X=C \times \Proj{1}$.

\begin{proposition}\label{proposition:BiholSplit}
Suppose that $S=C \times \Proj{1}$
is a connected complex surface, and that $G$
is a complex Lie group acting transitively,
faithfully and holomorphically on $S$.
Then either
\begin{enumerate}
\item from lemma~\vref{lemma:P1xP1}:
\(
S=\Proj{1} \times \Proj{1}
\)
and
\[
G = \Z{}_2 \ltimes 
\left(
\PSL{2,\C{}}
\times
\PSL{2,\C{}}
\right),
\]
or
\item
\(
G = G_0 \times \Bihol{\Proj{1}},
\)
where $G_0 \subset \Bihol{C}$
is a complex Lie group acting
holomorphically and transitively on $C=G_0/H_0$,
and thus $G_0/H_0$ is one of the examples
in table~\vref{table:OneD}.
\end{enumerate}
\end{proposition}
\begin{proof}
We need to argue that, besides the one
counterexample in the statement of this
proposition, all other complex Lie groups
$G$ acting holomorphically, faithfully and
transitively on a surface of the form
$S=C \times \Proj{1}$ must be subgroups
of $\Bihol{C} \times \Bihol{\Proj{1}}$.
By lemma~\vref{lemma:P1xP1},
this is true if $C=\Proj{1}$.
By lemma~\vref{lemma:BiholEP1},
it is true if $C$ is a compact
complex curve.

Suppose that $C=\C{}$.
By theorem~\vref{theorem:Lie}, looking
over the examples one-by-one,
we can see that the identity component
of $G$ is
\[
G^0 \subset \Bihol{C}^0 \times \Bihol{\Proj{1}}.
\]
Suppose that $g \in G$. Then $g$ must
preserve the Lie algebra $\LieG$ of $G$,
i.e. of $G^0$:
\[
\LieG \subset
\left(\C{} \ltimes \C{}\right) 
\oplus
\LieSL{2,\C{}}.
\]
A biholomorphism $g$ of $\C{} \times \Proj{1}$
preserves this Lie algebra if and only
if it lives in
\[
\left(\C{\times} \ltimes \C{}\right)
\times \PSL{2,\C{}},
\]
as required.

From Lie's classification, the only
other complex surface $S$ of the form
$C \times \Proj{1}$ which is complex-homogeneous
has $C=\C{\times}$. The
same computation shows that
preserving the same Lie algebra
requires once again a
biholomorphism in $\Bihol{C} \times \Bihol{\Proj{1}}$.
\end{proof}

\end{unabridged}

\section{Lifting from curves to surfaces}

\begin{unabridged}

\begin{example}
It is well known \cite{Biswas/McKay:2010}
that there is a unique holomorphic Cartan
geometry on $S=\Proj{2}$: it is the flat
Cartan geometry, i.e. $G/H$-structure, with
$G=\PSL{3,\C{}}$ and $G/H=\Proj{2}$.
\end{example}

\begin{example}
The surface $X=\Proj{1} \times \Proj{1}$
has biholomorphism group
\[
G = \Z{}_2 \ltimes
\left(\PSL{2,\C{}} \times \PSL{2,\C{}}\right).
\]
On $X$, there are precisely two holomorphic
foliations, the obvious horizontal
and vertical ones. The pair of them
is preserved by $G$, but neither one
is preserved individually.
This fact characterizes $X$ among
homogeneous surfaces.
\end{example}

\end{unabridged}

\begin{definition}
Suppose that $E \to M'$ is a holomorphic
Cartan geometry modelled on a complex
homogeneous space $G/H'$, with
Cartan connection $\omega$.
Suppose that
$H \subset H'$ is a closed complex Lie subgroup.
Let $M=E/H$. Then $E \to M$ is a holomorphic
Cartan geometry modelled on $G/H$, with
Cartan connection $\omega$. 
We say that the Cartan geometry on $M$ is the 
\emph{lift} 
of the Cartan geometry on $M'$.
\end{definition}

\begin{definition}\label{definition:lifting}
Suppose that $G$ is a complex Lie group and
that $H \subset H' \subset G$ are 
closed complex subgroups. Suppose that
$\dim G/H = 2$  and $\dim G/H' = 1$.
Suppose that $C$ is a complex curve.
Let $\pi = \fundgp{C}$.
Suppose that we have a $G/H'$-structure
on $C$, with developing map 
\map[\delta_C]{\tilde{C}}{G/H'},
and holonomy morphism
\map[h_C]{\pi}{G}.
Take the bundle $G/H \to G/H'$ and let
\[
\tilde{S} = \delta_C^*G/H,
\]
so that we have the diagram
\[
\xymatrix{
\tilde{S} \ar[d] \ar[r] & G/H \ar[d] \\
\tilde{C} \ar[d] \ar[r] & G/H' \\
C. \\
}
\]
So $\tilde{S}$ is a $\pi$-invariant
surface inside $\tilde{C} \times G/H$.
We then let $S=\pi \backslash \tilde{S}$,
so that $S$ is a complex surface, and 
\map{S}{C}
is a holomorphic fiber bundle, with fibers
isomorphic to $H'/H$.
The surface $\tilde{S}$ is a covering space
of $S$, and $S$ has a $G/H$-structure
with developing map
\map[\delta_S]{\tilde{S}}{G/H}
as above, and holonomy morphism
\map[h_S]{\fundgp{S}}{G}
given by the composition
\[
\fundgp{S} \to \pi \to G.
\]
We will say that this $G/H$-structure on $S$ is
the \emph{lift} of the $G/H'$-structure on $C$.
(This is a special case of a lift
of a Cartan geometry.)
\end{definition}

\begin{example}\label{example:liftCurve}
We define a lift of any structure
on a curve to a structure on any
flat $\Proj{1}$-bundle over that curve.
Suppose that $C$ is a compact complex
curve, and that $G_0/H_0$
is a complex-homogeneous curve. 
Let $\pi=\fundgp{C}$.
Take a $G_0/H_0$-structure on $C$,
say with developing map 
\[
\map[\delta_C]{\tilde{C}}{G_0/H_0}
\]
and holonomy morphism
\[
\map[h_C]{\pi}{G_0}.
\]
Take any group morphism
\[
\map[\rho]{\pi}{\PSL{2,\C{}}}.
\]
Let $\pi$ act
on $\tilde{C} \times \Proj{1}$
by
\[
\gamma \left(z,w\right)=
\left(\gamma \, z, \rho\left(\gamma\right)w\right).
\]
Define a compact complex surface $S$ by
\[
S = \tilde{C} \times_{\pi} \Proj{1}.
\]
Let $B \subset \PSL{2,\C{}}$
be the stabilizer of a point of $\Proj{1}$.
Let $G=G_0 \times \PSL{2,\C{}}$,
and $H=H_0 \times B$.
On $S$, define a $G/H$-structure,
by taking as developing map
\[
\map[\delta_S]{\tilde{S}=\tilde{C} \times \Proj{1}}%
{G/H=\left(G_0/H_0\right) \times \Proj{1}},
\]
the map
\[
\delta_S(z,w)=\left(\delta_C(z),w\right),
\]
and as holonomy morphism
\[
\map[h_S]{\fundgp{S}=\pi}{G=G_0 \times \PSL{2,\C{}}}
\]
the map
\[
h_S\left(\gamma\right)=\left(h_C\left(\gamma\right),
\rho\left(\gamma\right)\right).
\]
\end{example}

\begin{theorem}\label{theorem:RationalCurveClassification}
Suppose that $G/H$ is a connected
complex-homogeneous surface.
Suppose that $S$ is a compact complex surface
containing a rational curve
and that $S$ has an effective 
holomorphic Cartan geometry 
modelled on $G/H$.
Then the Cartan geometry is flat, a $G/H$-structure.
Up to isomorphism, either
\begin{enumerate}
\item
$S=\Proj{2}$ with the standard projective structure or
\item
$S=\Proj{1} \times \Proj{1}$ with the standard
$G/H$-structure where 
\[
G = \Z{}_2 \ltimes
\left(\PSL{2,\C{}} \times \PSL{2,\C{}}\right).
\]
or
\item
the $G/H$-structure on $S$, after perhaps
quotienting out the kernel of $G/H$, is constructed
as in example~\vref{example:liftCurve}.
The moduli space is then the product of
the moduli
space of $G_0/H_0$-structures on $C$
with the representation variety
\[
\Hom{\fundgp{C}}{\PSL{2,\C{}}}/\PSL{2,\C{}}.
\]
\end{enumerate}
%
\end{theorem}
\begin{proof}
The deformation space
of any rational curve in any compact
complex surface has
compact components.
By the main theorem
of Biswas and McKay \cite{Biswas/McKay:2010},
since our surface $S$ contains a rational
curve, there is some closed complex subgroup
$H' \subset G$ and a compact complex manifold $C$
with $\dim C < \dim S$,
and a fiber bundle morphism $S \to C$
with rational homogeneous fibers,
and a holomorphic Cartan geometry
on $C$ modelled on $G/H'$, so that
the holomorphic Cartan geometry
on $S$ is lifted from that on $C$.
Moreover, $H'/H$ must be a connected rational
homogeneous variety.
Every holomorphic Cartan geometry
on a complex curve is flat,
so is lifted from a unique $G/H'$-structure
with developing map and holonomy morphism.

If $C$ of dimension $0$, then $C$ 
is a point, $H'=G$, $S=G/H$
and the holomorphic Cartan geometry
on $S$ is the standard $G/H$-structure 
on $G/H$.
But then also we must have $H'/H$
a rational homogeneous variety, 
so $S=\Proj{1} \times \Proj{1}$
or $S=\Proj{2}$ with its standard
$G/H$-structure. 
If $S=\Proj{2}$, then
the group $G$, to act faithfully
and transitively, must be
$\PSL{3,\C{}}$ and
$H$ must be the Borel subgroup of 
$\PSL{3,\C{}}$. For
$S=\Proj{1} \times \Proj{1}$,
in order that $G$ act transitively
on $G/H$, we can have either
\[
G=\Z{}_2 \ltimes
\left( \PSL{2,\C{}} \times \PSL{2,\C{}}
\right)
\]
or the subgroup
\[
G=\PSL{2,\C{}} \times \PSL{2,\C{}}.
\]

Henceforth we can assume that $C$ is 
of dimension 1, so $H'/H=\Proj{1}$
and $G/H$ is a complex-homogeneous
surface, invariantly ruled.
By the classification of 
complex-homogeneous surfaces
\cite{Huckleberry:1986},
if a complex-homogeneous
surface is ruled, then it is a product
$G/H=\left(G_0/H_0\right) \times \Proj{1}$ 
where $G_0/H_0$ 
is a complex-homogeneous complex
curve, acted on transitively by $G$. {}%
\begin{unabridged}%
By proposition~\vref{proposition:BiholSplit},%
\end{unabridged}%
\begin{abridged}%
It is easy to check from the classification
of complex-homogeneous surfaces 
\cite{Mostow:1950,Olver:1995} that
\end{abridged}%
$G$ must be a product
\[
G = G_0 \times \PSL{2,\C{}},
\]
and so
\[
H = H_0 \times B,
\]
where $B$ is the Borel subgroup in $\PSL{2,\C{}}$.
The group $H'$ must contain $H$ and have $H'/H=\Proj{1}$,
so $H'$ must contain an image of $\SL{2,\C{}}$.
Since $\SL{2,\C{}}$ has no nontrivial morphism
to $\Bihol{C}$, we have 
$\left\{1\right\} \times \PSL{2,\C{}} \subset H'$, and 
therefore $H'=H_0 \times \PSL{2,\C{}}$.

The universal covering space of $S$ is
\begin{align*}
\tilde{S} 
&=
\delta_C^*\left(G/H\right)
\\
&=
\delta_C^*\left(\left(G_0/H_0\right) \times \Proj{1}\right)
\\
&=
\left(\delta_C^*\left(G_0/H_0\right)\right) \times \Proj{1}
\\
&=
\tilde{C} \times \Proj{1}.
\end{align*}
The developing map $\delta_S$ is then
by definition the identity map on $* \times \Proj{1}$.
\end{proof}

\begin{unabridged}

\subsection{Centralizers of subgroups of linear
fractional transformations}

\begin{definition}
Let's define some subgroups of $\PSL{2,\C{}}$.
Let $\C{\times}$ be the group of diagonal
matrices, $\C{}$ the group of translations
of the complex affine line, i.e. matrices of the form
\[
\begin{bmatrix}
1 & * \\
0 & 1
\end{bmatrix},
\]
$\C{\times} \ltimes \C{}$
the group of affine transformations of the 
complex affine line, i.e. matrices of the form
\[
\begin{bmatrix}
a & b \\
0 & \frac{1}{a}
\end{bmatrix},
\]
$C_2$ the group consisting
of
\[
I, 
\begin{bmatrix}
i & 0 \\
0 & -i
\end{bmatrix},
\]
$C_2 \times C_2$ the group consisting of
\[
I, 
\begin{bmatrix}
i & 0 \\
0 & -i
\end{bmatrix},
\begin{bmatrix}
0 & i \\
i & 0
\end{bmatrix},
\begin{bmatrix}
0 & -1 \\
1 & 0
\end{bmatrix},
\]
and $C_2 \times \C{\times}$ the group of
matrices of the form
\[
\begin{bmatrix}
a & 0 \\
0 & \frac{1}{a}
\end{bmatrix}
\text{ or }
\begin{bmatrix}
0 & b \\
-\frac{1}{b} & 0
\end{bmatrix}.
\]
Any regular polyhedron has a group
of rotations lying in 
$\SO{3} \subset \PSL{2,\C{}}$; 
we won't need notation for these
groups or for the cyclic and
dihedral groups of rotations.
There are also the various
groups $C_n \ltimes \C{}$
of matrices of the form
\[
\begin{bmatrix}
a & b \\
0 & \frac{1}{a}
\end{bmatrix}
\]
with $a^n=1$. 
\end{definition}

\begin{remark}
It is well
known (see Huckleberry \cite{Huckleberry:1986}
p. 324)
that every closed complex 
subgroup of $\PSL{2,\C{}}$
is conjugate to a unique one of the
subgroups listed above.
\end{remark}

\begin{table}
\[
\begin{array}{cccc}
\textbf{Group element} & 
\textbf{Centralizer} 
\betweenEntries
\begin{bmatrix}
1 & 0 \\
0 & 1
\end{bmatrix}
&
\PSL{2,\C{}} 
\betweenEntries
\begin{bmatrix}
i & 0 \\
0 & -i
\end{bmatrix}
&
C_2 \times \C{\times}
\betweenEntries
\begin{bmatrix}
\alpha & 0 \\
0 & \frac{1}{\alpha}
\end{bmatrix}
&
\C{\times}
\\
\alpha \ne \pm 1, \pm i
\betweenEntries
\begin{bmatrix}
1 & \alpha \\
0 & 1
\end{bmatrix}
&
\C{}
\\
\alpha \ne 0
\betweenEntries
\begin{bmatrix}
0 & i \\
i & 0
\end{bmatrix}
&
\begin{bmatrix}
a & b \\
\varepsilon b & \varepsilon a
\end{bmatrix}
\\
&
\varepsilon=\pm 1,
a^2 - b^2 = \varepsilon.
\end{array}
\]
\caption{The centralizers of various elements
in $\PSL{2,\C{}}$.}%
\label{table:Centralizers}
\end{table}

\begin{table}
\[
\begin{array}{cccc}
\textbf{Group elements} & 
\textbf{Centralizer} 
\betweenEntries
\begin{bmatrix}
\alpha & 0 \\
0 & \frac{1}{\alpha}
\end{bmatrix}
\ne 
\begin{bmatrix}
\beta & 0 \\
0 & \frac{1}{\beta}
\end{bmatrix}
&
\C{\times}
\betweenEntries
\begin{bmatrix}
1 & \alpha \\
0 & 1
\end{bmatrix}
\ne 
\begin{bmatrix}
1 & \beta  \\
0 & 1
\end{bmatrix}
&
\C{}
\betweenEntries
C_2 \times C_2
&
C_2 \times C_2
\end{array}
\]
\caption{The centralizers of the subgroups
in $\PSL{2,\C{}}$ which are generated
by two commuting elements, but not by a
single element.}%
\label{table:PairCentralizers}
\end{table}

\begin{lemma}\label{lemma:CentralizerList}
Suppose that $\Gamma \subset \PSL{2,\C{}}$
is a subgroup and let $Z$ be the centralizer 
and $N$ the normalizer of $\Gamma$ 
in $\PSL{2,\C{}}$.
Then $\Gamma=\left\{I\right\}$ if and only if $Z=N=\PSL{2,\C{}}$. If $\Gamma=C_2$ then 
$Z=N=C_2 \times \C{\times}$.
If $\Gamma \subset \C{\times}$ and $\Gamma \ne \left\{I\right\}$
and $\Gamma \ne C_2$
then $Z=\C{\times}$ and $N=C_2 \times \C{\times}$.
If $\Gamma \subset \C{}$  and $\Gamma\ne \left\{I\right\}$ 
then $Z=\C{}$ and $N$ is the group of 
matrices of the form
\[
\begin{bmatrix}
a & b \\
0 & \frac{1}{a}
\end{bmatrix}
\]
so that $a^2 \Gamma = \Gamma \subset \C{}$.
If $\Gamma=C_2 \times C_2$ then $Z=C_2 \times C_2$
and $N$ is the group of rotations of the cube.
If $\Gamma$ is not conjugate to one of these
subgroups then $Z=\left\{1\right\}$.
\end{lemma}
\begin{proof}
If $\Gamma$ is generated by one element
then, up to conjugacy, $Z$ is listed in table~\vref{table:Centralizers}.
If $\Gamma$ is generated by two commuting elements,
then $Z$ is listed in table~\vref{table:PairCentralizers}.
Let us now consider the general case.
Clearly $\Gamma$ must lie in
the centralizer of $Z$.
The group $\Gamma$ is trivial just when
$Z=\PSL{2,\C{}}$.
The group $Z$ is a Zariski closed 
subgroup of $\PSL{2,\C{}}$.
Up to conjugacy, there is only one 2-dimensional 
complex subgroup of $\PSL{2,\C{}}$, the affine
group of $\C{}$.
If $Z$ contains the affine group, then 
every element of
$\Gamma$ commutes with all diagonal and 
all strictly
upper triangular matrices, and so $\Gamma$ 
is trivial,
and $Z=\PSL{2,\C{}}$.
Similarly if $Z$ contains a conjugate of the affine
group. So we can assume that $Z$ has dimension
at most 1.
If $Z$ has dimension 1, then 
$Z$ must contain one of the 1-parameter
subgroups conjugate to the diagonal matrices
or the strictly upper triangular matrices.
But then every element of $\Gamma$
must commute with all of these things,
so must lie in one of these 1-parameter subgroups.

Suppose that the centralizer $Z$ has dimension $0$.
So $Z$ is one of the finite subgroups of the rotation
group: the trivial group, the cyclic groups of rotations around
an axis, the dihedral groups of rotations
around an axis together with a rotation
that reverses that axis, the
symmetry groups of the tetrahedron,
cube and icosahedron. 
The centralizer of each cyclic group
is the diagonal matrices, except for the
cyclic group $C_2$,
which we have already discussed.
The centralizer of the dihedral group $D_2 = C_2 \times C_2$
is $C_2 \times C_2$.
The centralizers of the other dihedral 
groups are trivial.
If the rotation 
symmetries of a regular polyhedron
commute with a particular linear fractional
transformation, then the (one or two) fixed points
of that linear fractional transformation
must be permuted by those rotations,
which is clearly impossible.
Therefore the centralizers of the rotation
symmetries of the regular polyhedra are trivial. 
\end{proof}

\subsection{Automorphism groups}

\begin{proposition}[Loray and 
Mar\'in \cite{Loray/MarinPerez:2009}]
Let $\Lambda \subset \C{}$ be a lattice,
and
\[
\map[\rho]{\Lambda}{\PSL{2,\C{}}}
\]
a representation of $\Lambda$.
Let $C=\C{}/\Lambda$ and 
let $S \to C$ be the flat $\Proj{1}$-bundle
\[
S = \C{} \times_{\rho} \Proj{1}.
\]
Up to conjugacy, either
\begin{enumerate}
\item $\rho\left(\Lambda\right) \subset \C{\times}$ 
or
\item $\rho\left(\Lambda\right) \subset \C{}$
or
\item $\rho\left(\Lambda\right)=C_2 \times C_2$.
\end{enumerate}

The bundle $S \to C$ is trivial
if and only if there is a
complex 1-parameter subgroup
$\map{\C{}}{\PSL{2,\C{}}}$ 
extending $\rho$.
For example, if $\rho\left(\Lambda\right) \subset \C{\times}$ then $S \to C$ is a trivial bundle 
if and
only if there is a constant $b \in \C{}$ so that
\[
\rho(\lambda)
=
\begin{bmatrix}
e^{b\lambda/2} & 0 \\
0 & e^{-b \lambda/2}
\end{bmatrix}.
\]
If $\rho\left(\Lambda\right) \subset \C{}$
then $S \to C$ is 
a trivial bundle if and
only if there is a constant $b \in \C{}$ so that
\[
\rho(\lambda)
=
\begin{bmatrix}
1 & b\lambda \\
0 & 1
\end{bmatrix}.
\]
Otherwise $S \to C$ is not a trivial bundle.
\end{proposition}

\begin{example}
Suppose that $M$ is a rational homogeneous variety
with a holomorphic Cartan geometry.
Then clearly the geometry is that of the
model $G/H$, and globally homogeneous.
\end{example}

\subsubsection{Automorphisms and the centralizer}

\begin{example}
Suppose that $S$ is a compact complex surface
containing a rational curve and bearing
a holomorphic effective Cartan geometry.
Asssume that $S$ is not a rational
homogeneous variety.
By 
theorem~\vref{theorem:RationalCurveClassification},
we can assume that $S$ is a flat $\Proj{1}$-bundle
$S \to C$ and that $C$ is a compact complex curve,
and that the geometry on $S$ is constructed
from a $G_0/H_0$-structure on $C$ as
in example~\vref{example:liftCurve}.
If $C$ is a rational curve,
then any flat $\Proj{1}$-bundle on $C$ is trivial,
so the geometry is globally homogeneous.
We can assume that $C$ has genus at least 
one. The only rational curves in $S$ 
lie in the fibers of $S \to C$.
Every biholomorphism of $S$ preserves
the fibration. The automorphism group
of the geometry is a closed subgroup of the
biholomorphism group of $S$. Because
$S$ is a $\Proj{1}$-bundle over a curve,
$S$ is an algebraic surface
\cite{Barth/Hulek/Peters/VanDeVen:2004}, so the
automorphism group is an algebraic group 
\cite{Serre:1955}.

The universal covering space is
\(\tilde{S} 
=
\tilde{C} \times \Proj{1}\).
Write points of $\tilde{S}$
as $(z,w)$. 
Pick any automorphism $\tilde\psi$
of the structure on $C$, i.e.
\[
\delta_C \circ \tilde\psi
=
g_0 \delta_C,
\]
for some $g_0 \in G_0$. Take any
$h_0 \in Z$ in the centralizer
$Z$ of the image of $\rho$.
Then let 
\[
\tilde\varphi(z,w)=\left(\tilde\psi(z),h_0 \, w\right).
\]
Compute that
\[
\delta_S \circ \tilde\varphi
=
\left(g_0, h_0\right)
\delta_S.
\]

Conversely, pick some lift 
\[
\map[\tilde\varphi]{\tilde{S}}{\tilde{S}}
\]
of an automorphism $\varphi$.
Since the automorphism must act
on $C$ as an automorphism of its
structure,
\[
\tilde\varphi(z,w) = 
\left(\tilde\psi(z),g(z)w\right),
\]
for some automorphism
\[
\map[\tilde\psi]{\tilde{C}}{\tilde{C}},
\]
and a map
\[
\map[g]{\tilde{C}}{\PSL{2,\C{}}}.
\]
To preserve the geometry,
there must be
some $\left(g_0,h_0\right) \in 
G=G_0 \times \PSL{2,\C{}}$ so that
\[
\delta_S \circ \tilde\varphi =
\left(g_0,h_0\right)
\delta_S.
\]
Expanding out, we find that \(g(z)=h_0\)
is constant. 
Let $\pi=\fundgp{C}=\fundgp{S}$.
In order that $\tilde\varphi$
be the lift of a map $\varphi$ on $S$,
we need to have, for every $\gamma \in \pi$,
some $\gamma' \in \pi$ so that
\[
\tilde\varphi \circ \gamma = \gamma' \circ \tilde\varphi,
\]
so
\[
\tilde\psi \circ \gamma = \gamma' \circ \tilde\psi,
\]
and
\[
h_0 \, \rho\left(\gamma\right) h_0^{-1} 
= \rho\left(\gamma'\right).
\]

If our automorphism $\tilde\psi$ belongs
to the identity component of the automorphism
group of the structure on $C$, then
$\gamma'=\gamma$ and
\[
h_0 \, \rho\left(\gamma\right) h_0^{-1} 
= \rho\left(\gamma\right)
\]
for every $\gamma \in \pi$. Therefore
$h_0 \in Z$.
Write $\Aut{E}$ for the automorphism
group of the Cartan geometry on $S$,
and $\Aut{E_C}$ for the automorphism
group of the structure on $C$.
Let $\AutP{E_C} \subset \Aut{E_C}$ and 
$\AutP{E} \subset \Aut{E}$ 
be the subgroups of automorphisms
acting trivially on $\pi=\fundgp{C}$.
So $\AutP{E}=Z \times \AutP{E_C}$,
where $Z$ is the centralizer
of the image of $\rho$.
Thus the exact sequence
\[
1 \to Z \times \AutP{E_C} \to \Aut{E} \to
\pi.
\]
\end{example}

\subsubsection{Parallel sections}

\begin{example}
Continuing with the last example, 
the surface $S$ is a flat $\Proj{1}$-bundle 
$S \to C$,
given by the quotient 
\[
S = \C{} \times_{\pi} \Proj{1},
\]
by the action
\[
\gamma\left(z,w\right)=
\left(\gamma \, z,\rho\left(\gamma\right)w\right).
\]
If $w_0 \in \Proj{1}$ is invariant under $\rho$,
then the points $\left(z,w_0\right)$ for 
$z \in \tilde{C}$
form a parallel section of $S \to C$.
If $\rho$ has nontrivial image in the
diagonal matrices then $Z$ 
is the diagonal matrices, and
there are two parallel sections
of $S \to C$, and they are invariant under
$Z$ and so under $\AutP{E}$.
If $\rho$ has nontrivial image in the
strictly upper triangular matrices, then
$Z$ is the strictly upper triangular
matrices, and there is one parallel
section of $S \to C$, and it is 
invariant under $Z$ and so under $\AutP{E}$.
Otherwise (for $\rho$ having image
not trivial or not conjugate
to a subgroup of diagonal or 
strictly upper triangular matrices)
the group $Z$ is trivial, and so
the group $\AutP{E}$ is just the finite
group $\AutP{E_C}$.
For example, if $\rho$ has a finite orbit, 
then that orbit represents a parallel 
multivalued section, i.e.
a parallel finite unbranched covering of $C$
sitting in $S$. Each of these parallel
multivalued 
sections is invariant under $\AutP{E}$.
However, some of them might be permuted
by $\Aut{E}$.
\end{example}

\begin{example}
Continuing with the previous example,
suppose that the base $C$ has genus $g \ge 2$.
Then $C$ has a finite biholomorphism
group. The model must be $G_0=\PSL{2,\C{}}$,
$X_0=\Proj{1}$, from table~\vref{table:OneD},
and the automorphism group extends
the centralizer $Z$
\[
1 \to Z \to \Aut{E} \to F \to 1,
\]
by the finite group $F$ of automorphisms
of $C$ that preserve the quadratic differential.
The group $Z$ can be any of those
listed in lemma~\vref{lemma:CentralizerList}.
Given any finite subgroup $F$ of the automorphism
group of $C$, we can average the quadratic
differentials over $F$, to produce $F$-invariant
quadratic differentials. However it is possible
that every $F$-invariant
quadratic differential is stabilized by
some larger group than $F$, so there is very
little we can say about the classification
of automorphism groups of holomorphic
Cartan geometries on higher genus curves.
\end{example}

\begin{example}
Next suppose that $C$ is an elliptic curve $C=\C{}/\Lambda$.
If $\rho$ is valued in diagonal matrices, then
there are two parallel sections: $w_0=0$ and $w_0=\infty$.
\end{example}

\begin{example}
Again suppose that $C$ is an elliptic curve $C=\C{}/\Lambda$.
Suppose that 
$\rho$ is valued in strictly upper triangular
matrices and $\rho$ is not trivial. 
There is one parallel section: $w=\infty$.
\end{example}

\begin{example}\label{example:ClosedOrbits}
Again suppose that $C$ is an elliptic curve $C=\C{}/\Lambda$.
If $\rho\left(\Lambda\right)=C_2 \times C_2$,
then there is no parallel section. However,
the 3 subsets of $\left(z,w\right)$ with
$w$ on opposite sides of the cube,
i.e. 
\begin{enumerate}
\item
$w=0$ and $w=\infty$, or
\item
$w=1$ and $w=-1$, or
\item
$w=i$ and $w=-i$,
\end{enumerate}
form three elliptic curves in $S$,
each a 2-1 covering of $C$. 
Any other point $w \in \Proj{1}$ belongs to 
an orbit $w,-w,\frac{1}{w}, -\frac{1}{w}$
of order 4, giving a parallel elliptic
curve in $S$ which is a 4-1 covering of $C$.
Since the normalizer of $C_2 \times C_2$
is the group of rotations of the cube, 
the automorphism group can at most permute
these 3 different 2-1 coverings and also permute
6 of the these different 4-1 coverings.
In particular, each automorphism group orbit
is possibly only a finite set, and
is at most 6 different parallel elliptic 
curves; there are no open orbits.
\end{example}

\subsubsection{The rotations of the cube}

Suppose that $S$ is a compact complex
surface with a holomorphic effective 
$G/H$-geometry for some complex Lie
group $G$ and closed complex subgroup 
$H \subset G$.
Suppose that $S \to C$ is a flat $\Proj{1}$-bundle
Again suppose that $C$ 
is an elliptic curve $C=\C{}/\Lambda$.
Define a cyclic group by
\[
C_n = \Aut{E_C}/\AutP{E_C}.
\]
This group is generated by some $n$-th root of 1,
say $\alpha$, acting on $\Lambda \subset \C{}$.
Moreover, $n=1, 2, 3, 4$, or $6$.
Every automorphism $\psi \in \Aut{E_C}$,
say mapping to $\alpha^k \in C_n$,
extends to an automorphism
\[
\tilde{\phi}\left(z,w\right)
=
\left(\alpha^k z,h_0 w\right)
\]
for every $h_0 \in \PSL{2,\C{}}$ so that
\[
h_0 \rho\left(\lambda\right) h_0^{-1} = \rho\left(\alpha^k \lambda\right).
\]

\begin{example}
If $n$ is even (i.e. $n=2$ or $n=4$ or $n=6$), 
so $-1 \in C_n$, and
$\rho\left(\Lambda\right) \subset \C{\times}$,
let
\[
h_0 =
\begin{bmatrix} 
0 & i \\
i & 0
\end{bmatrix}.
\]
Then
\[
\tilde\phi(z,w)=\left(-z,h_0 w\right)
\]
is an automorphism permuting the
two parallel sections $w=0$ and $w=\infty$.
Since the normalizer of the diagonal
matrices is $C_2 \times \C{\times}$,
if $\rho\left(\Lambda\right) \ne \left\{I\right\},
C_2$, and $n$ is even, then
there are two group orbits: the
two parallel sections form one
orbit, and the complement
of those two sections forms the other.
\end{example}

\begin{example}
Suppose that 
$\rho\left(\Lambda\right)=C_2 \times C_2$.
Clearly $h_0 \in \PSL{2,\C{}}$ 
conjugates these various $C_2 \times C_2$
matrices to one another just when
$h_0$ belongs to the group of
rotations of the cube. 

Suppose that our group $C_n$ has
$n \ne 2$, so $n=3,4$ or $6$,
say generated by $\alpha \in \C{\times}$.
Take $\lambda, \mu \in \Lambda$
any two generators of $\Lambda$.
So $\rho\left(\lambda\right)$
and $\rho\left(\mu\right)$ must 
generate $C_2 \times C_2$. 
Therefore they must be rotations
by angle $\pi$ around two of three 
perpendicular axes.
The elements $\rho\left(\alpha \lambda\right)$
and $\rho\left(\alpha \mu\right)$
must also be generators of the same 
$C_2 \times C_2$, so also rotations
by angle $\pi$ around two of the same three 
perpendicular axes.
We can then let $h_0$ be any rotation
which matches up the first pairs of
axes with those the second pair.
Then
\[
\tilde{\phi}\left(z,w\right)
=\left(\alpha z, h_0 w\right)
\]
is an automorphism. But then
$Z=C_2 \times C_2$ also acts
as automorphisms, so the automorphism
group is 
\[
\Aut{E} = 
C_2 \times C_2 \times 
\left(C_n \ltimes \C{}\right).
\]
The orbits are as described
in example~\vref{example:ClosedOrbits}.
\end{example}

\end{unabridged}

\section{Conclusion}

We have
found that for compact complex surfaces
which contain rational curves, essentially
there are no holomorphic Cartan
geometries, except for a trivial construction
from holomorphic 
locally homogeneous geometric
structures on curves. 
On compact complex surfaces
without rational curves,
and in higher dimensions,
more complicated phenomena
show up.

In our examples above, we see that the
moduli stack of holomorphic Cartan
geometries on a fixed complex manifold,
with a fixed model $G/H$, might not
be a complex analytic space or orbispace. We 
can also
see that deformations of a holomorphic
Cartan geometry can give rise
to nontrivial deformations of the total
space of the Cartan geometry as
a holomorphic principal bundle:
take the affine structure of an elliptic
curve, and a family of representations
$\rho_t$ of its fundamental
group into $\PSL{2,\C{}}$,
so that $\rho_{t_1}$
and $\rho_{t_2}$ are not conjugate
for generic values $t_1$
and $t_2$.  Build a holomorphic Cartan geometry
on each of the associated 
flat $\Proj{1}$-bundles.

\bibliographystyle{amsplain}
\bibliography{toothpick}

\end{document}